\documentclass[12pt]{amsart}
\usepackage{stmaryrd}
\usepackage{}
\usepackage{amsmath}
\usepackage{amsfonts}
\usepackage{amssymb}
\usepackage[all,cmtip]{xy}           
\usepackage{shuffle}
\usepackage{bbding}
\usepackage{txfonts}
\usepackage[shortlabels]{enumitem}
\usepackage{ifpdf}
\ifpdf
\usepackage[colorlinks,final,backref=page,hyperindex]{hyperref}
\else
\usepackage[colorlinks,final,backref=page,hyperindex,hypertex]{hyperref}
\fi
\usepackage{tikz}
\usepackage[active]{srcltx}

\makeatletter

\newtheorem{theorem}{Theorem}[section]
\newtheorem{prop}[theorem]{Proposition}
\newtheorem{lemma}[theorem]{Lemma}
\newtheorem{coro}[theorem]{Corollary}
\newtheorem{prop-def}{Proposition-Definition}[section]

\theoremstyle{definition}
\newtheorem{defn}[theorem]{Definition}

\newtheorem{remark}[theorem]{Remark}
\newtheorem{exam}[theorem]{Example}

\newcommand{\nc}{\newcommand}

\newcommand {\emptycomment}[1]{}

\nc{\delete}[1]{{}}
\nc{\mmargin}[1]{}

\nc{\mlabel}[1]{\label{#1}}  
\nc{\mcite}[1]{\cite{#1}}  
\nc{\mref}[1]{\ref{#1}}  
\nc{\meqref}[1]{\eqref{#1}}  
\nc{\mbibitem}[1]{\bibitem{#1}} 

\delete{
	\nc{\mlabel}[1]{\label{#1}  
		{\hfill \hspace{1cm}{\bf{{\ }\hfill(#1)}}}}
	\nc{\mcite}[1]{\cite{#1}{{\bf{{\ }(#1)}}}}  
	\nc{\mref}[1]{\ref{#1}{{\bf{{\ }(#1)}}}}  
	\nc{\meqref}[1]{\eqref{#1}{{\bf{{\ }(#1)}}}}  
	\nc{\mbibitem}[1]{\bibitem[\bf #1]{#1}} 
}

\newcommand{\g}{\mathfrak g}

\newcommand{\bk}{{\mathbf{k}}}

\nc{\splie}{{\rhd\!\!\rhd}}

\nc{\vep}{\varepsilon}
\nc{\bin}[2]{ (_{\stackrel{\scs{#1}}{\scs{#2}}})}  
\nc{\binc}[2]{(\!\! \begin{array}{c} \scs{#1}\\
		\scs{#2} \end{array}\!\!)}  
\nc{\bincc}[2]{  ( {\scs{#1} \atop
		\vspace{-1cm}\scs{#2}} )}  
\nc{\oline}[1]{\overline{#1}}
\nc{\mapm}[1]{\lfloor\!|{#1}|\!\rfloor}
\nc{\bs}{\bar{S}}
\nc{\cast}{{\,\mbox{\raisebox{.8pt}{$\scriptstyle \circledast$}}\,}}
\nc{\la}{\longrightarrow}
\nc{\ot}{\otimes}
\nc{\rar}{\rightarrow}
\nc{\dar}{\downarrow}
\nc{\dap}[1]{\downarrow \rlap{$\scriptstyle{#1}$}}
\nc{\defeq}{\stackrel{\rm def}{=}}
\nc{\dis}[1]{\displaystyle{#1}}
\nc{\dotcup}{\ \displaystyle{\bigcup^\bullet}\ }
\nc{\hcm}{\ \hat{,}\ }
\nc{\hts}{\hat{\otimes}}
\nc{\hcirc}{\hat{\circ}}
\nc{\lleft}{[}
\nc{\lright}{]}
\nc{\llb}{\llbracket}
\nc{\rrb}{\rrbracket}
\nc{\curlyl}{\left \{ \begin{array}{c} {} \\ {} \end{array}
	\right .  \!\!\!\!\!\!\!}
\nc{\curlyr}{ \!\!\!\!\!\!\!
	\left . \begin{array}{c} {} \\ {} \end{array}
	\right \} }
\nc{\longmid}{\left | \begin{array}{c} {} \\ {} \end{array}
	\right . \!\!\!\!\!\!\!}
\nc{\ora}[1]{\stackrel{#1}{\rar}}
\nc{\ola}[1]{\stackrel{#1}{\la}}
\nc{\scs}[1]{\scriptstyle{#1}} \nc{\mrm}[1]{{\rm #1}}
\nc{\dirlim}{\displaystyle{\lim_{\longrightarrow}}\,}
\nc{\invlim}{\displaystyle{\lim_{\longleftarrow}}\,}
\nc{\dislim}[1]{\displaystyle{\lim_{#1}}} \nc{\colim}{\mrm{colim}}
\nc{\mvp}{\vspace{0.3cm}} \nc{\tk}{^{(k)}} \nc{\tp}{^\prime}
\nc{\ttp}{^{\prime\prime}} \nc{\svp}{\vspace{2cm}}
\nc{\vp}{\vspace{8cm}}
\nc{\modg}[1]{\!<\!\!{#1}\!\!>}
\nc{\intg}[1]{F_C(#1)}
\nc{\lmodg}{\!<\!\!}
\nc{\rmodg}{\!\!>\!}
\nc{\cpi}{\widehat{\Pi}}
\nc{\labs}{\mid\!}
\nc{\rabs}{\!\mid}
\nc{\btr}{\blacktriangleright}

\nc{\ad}{\mrm{ad}}
\nc{\rRB}{\mathsf{rRB}}
\nc{\cocrRB}{\mathsf{cocrRB}}
\nc{\PH}{\mathsf{PH}}
\nc{\cocPH}{\mathsf{cocPH}}
\nc{\ann}{\mrm{ann}}
\nc{\Aut}{\mrm{Aut}}
\nc{\Der}{\mrm{Der}}
\nc{\Sym}{\mrm{Sym}}
\nc{\br}{\mrm{bre}}
\nc{\can}{\mrm{can}}
\nc{\Cont}{\mrm{Cont}}
\nc{\rchar}{\mrm{char}}
\nc{\cok}{\mrm{coker}}
\nc{\de}{\mrm{dep}}
\nc{\dtf}{{R-{\rm tf}}}
\nc{\dtor}{{R-{\rm tor}}}

\nc{\Dif}{\mrm{Diff}}
\nc{\Div}{\mrm{Div}}
\nc{\End}{\mrm{End}}
\nc{\Ext}{\mrm{Ext}}
\nc{\Fil}{\mrm{Fil}}
\nc{\Fr}{\mrm{Fr}}
\nc{\Frob}{\mrm{Frob}}
\nc{\Gal}{\mrm{Gal}}
\nc{\GL}{\mrm{GL}}
\nc{\Gr}{\mrm{Gr}}
\nc{\Hom}{\mrm{Hom}}
\nc{\Hoch}{\mrm{Hoch}}
\nc{\hsr}{\mrm{H}}
\nc{\hpol}{\mrm{HP}}
\nc{\id}{\mrm{id}}
\nc{\im}{\mrm{im}}
\nc{\inv}{\mrm{inv}}
\nc{\Id}{\mrm{Id}}
\nc{\ID}{\mrm{ID}}
\nc{\Irr}{\mrm{Irr}}
\nc{\incl}{\mrm{incl}}
\nc{\length}{\mrm{length}}
\nc{\NLSW}{\mrm{NLSW}}
\nc{\Lie}{\mrm{Lie}}
\nc{\mchar}{\rm char}
\nc{\mpart}{\mrm{part}}
\nc{\ql}{{\QQ_\ell}}
\nc{\qp}{{\QQ_p}}
\nc{\rank}{\mrm{rank}}
\nc{\rcot}{\mrm{cot}}
\nc{\rdef}{\mrm{def}}
\nc{\rdiv}{{\rm div}}
\nc{\rtf}{{\rm tf}}
\nc{\rtor}{{\rm tor}}
\nc{\res}{\mrm{res}}
\nc{\SL}{\mrm{SL}}
\nc{\Spec}{\mrm{Spec}}
\nc{\tor}{\mrm{tor}}
\nc{\Tr}{\mrm{Tr}}
\nc{\tr}{\mrm{tr}}
\nc{\wt}{\mrm{wt}}

\nc{\bfk}{{\bf k}}
\nc{\bfone}{{\bf 1}}
\nc{\bfzero}{{\bf 0}}
\nc{\detail}{\marginpar{\bf More detail}
	\noindent{\bf Need more detail!}
	\svp}
\nc{\gap}{\marginpar{\bf Incomplete}\noindent{\bf Incomplete!!}
	\svp}
\nc{\FMod}{\mathbf{FMod}}
\nc{\Int}{\mathbf{Int}}
\nc{\Mon}{\mathbf{Mon}}
\nc{\remarks}{\noindent{\bf Remarks: }}
\nc{\Rep}{\mathbf{Rep}}
\nc{\Rings}{\mathbf{Rings}}
\nc{\Sets}{\mathbf{Sets}}
\nc{\Diff}{\mathbf{Diff}}
\nc{\Inte}{\mathbf{Inte}}
\nc{\U}{\mathrm{U}}
\nc{\uu}{\mathrm{u}}

\nc{\BA}{{\mathbb A}}   \nc{\CC}{{\mathbb C}}
\nc{\DD}{{\mathbb D}}   \nc{\EE}{{\mathbb E}}
\nc{\FF}{{\mathbb F}}   \nc{\GG}{{\mathbb G}}
\nc{\HH}{{\mathbb H}}   \nc{\LL}{{\mathbb L}}
\nc{\NN}{{\mathbb N}}   \nc{\PP}{{\mathbb P}}
\nc{\QQ}{{\mathbb Q}}   \nc{\RR}{{\mathbb R}}
\nc{\TT}{{\mathbb T}}   \nc{\VV}{{\mathbb V}}
\nc{\ZZ}{{\mathbb Z}}   \nc{\TP}{\widetilde{P}}

\nc{\cala}{{\mathcal A}}    \nc{\calc}{{\mathcal C}}
\nc{\cald}{\mathcal{D}}     \nc{\cale}{{\mathcal E}}
\nc{\calf}{{\mathcal F}}    \nc{\calg}{{\mathcal G}}
\nc{\calh}{{\mathcal H}}    \nc{\cali}{{\mathcal I}}
\nc{\call}{{\mathcal L}}    \nc{\calm}{{\mathcal M}}
\nc{\caln}{{\mathcal N}}    \nc{\calo}{{\mathcal O}}
\nc{\calp}{{\mathcal P}}    \nc{\calr}{{\mathcal R}}

\nc{\cals}{{\mathcal S}}    \nc{\calt}{{\Omega}}
\nc{\calv}{{\mathcal V}}    \nc{\calw}{{\mathcal W}}
\nc{\calx}{{\mathcal X}}

\nc{\fraka}{{\mathfrak a}}
\nc{\frakb}{\mathfrak{b}}
\nc{\frakg}{{\frak g}}
\nc{\frakl}{{\frak l}}
\nc{\fraks}{{\frak s}}
\nc{\frakB}{{\frak B}}
\nc{\frakm}{{\frak m}}
\nc{\frakM}{{\frak M}}
\nc{\frakp}{{\frak p}}
\nc{\frakW}{{\frak W}}
\nc{\frakX}{{\frak X}}
\nc{\frakS}{{\frak S}}
\nc{\frakA}{{\frak A}}
\nc{\frakx}{{\frakx}}

\newcommand{\lieg}{\mathfrak{g}} 
\newcommand{\lieh}{\mathfrak{h}} 
\newcommand{\liel}{\mathfrak{l}} 
\newcommand{\liec}{\mathfrak{c}}

\newcommand{\ra}{\rightarrow} 

\newcommand{\on}[1]{\operatorname{#1}}

\newcommand{\yn}[1]{\textcolor{orange}{#1 }}
\nc{\ynr}[1]{\textcolor{orange}{\underline{Yunnan:}#1 }}

\nc{\lir}[1]{\textcolor{red}{\underline{Li:}#1 }}

\delete{
	\input cyracc.def

	\newtheorem{theorem}{Theorem}[section]
	\newtheorem{lemma}[theorem]{Lemma}

	\theoremstyle{definition}

	\theoremstyle{remark}
	\newtheorem{remark}[theorem]{Remark}

	\allowdisplaybreaks
	
	\numberwithin{equation}{section}
	
}

\begin{document}

\title[On restricted Rota-Baxter Lie algebras of arbitrary weight]
{On restricted Rota-Baxter Lie algebras of arbitrary weight}

\author{Yunnan Li}
\address{School of Mathematics and Information Science, Guangzhou University,
Guangzhou 510006, China}
\email{ynli@gzhu.edu.cn}

\author{Ke Ou}
\address{School of Statistics and Mathematics, Yunnan University of Finance and Economics, Kunming 650221, China}
\email{keou@ynufe.edu.cn}

\begin{abstract}
Recently, Ehret and Gilliers introduced the notion of a (trivially) restricted post-Lie algebra, {recovering} the concepts of a restricted Lie algebra and a restricted pre-Lie algebra.
In this paper, we specifically introduce restricted Rota-Baxter Lie algebras of arbitrary weight {with an intrinsic graph subalgebra characterization}. We show that, via the splitting property, they give rise to restricted post-Lie algebras, and furthermore possess a novel replication property.
We then present two natural constructions of such restricted Rota-Baxter structures in prime characteristic: one arising from Rota-Baxter associative algebras of arbitrary weight, and the other from Rota-Baxter Lie algebras of weight $1$ .
The Rota-Baxter $p$-envelopes of a Rota-Baxter Lie algebra are also examined.
\end{abstract}

\keywords{restricted Lie algebra, Rota-Baxter Lie algebra, post-Lie algebra\\
\qquad 2020 Mathematics Subject Classification. 17B50, 17B38, 17D25}

\maketitle

\tableofcontents

\allowdisplaybreaks

\section{Introduction}

Restricted Lie algebras, a fundamental structure in the theory of Lie algebras over fields of positive characteristic, were first introduced and systematically studied by N. Jacobson in 1937 (\cite{Ja1}). Jacobson's pioneering work laid the foundation for understanding how the additional $p$-map (denoted $x \mapsto x^{[p]}$) naturally obtained from the Frobenius map in associative algebras of characteristic $p > 0$.
The significance of restricted Lie algebras cannot be overstated. They play a predominant role in modular Lie algebra theory, intimately connecting the simple Lie algebras with algebraic groups, Hopf algebras, and representation theory. Crucially, finite-dimensional Lie algebras that arise ``in nature'' are typically restricted. Classical examples include the {commutator} Lie algebra {or the derivation algebra} associated with any associative algebra,
the Lie algebra of an algebraic group, and the primitive Lie subalgebra of an Hopf algebra. This ubiquity, coupled with technical tools like the Jordan-Chevalley-Seligman decomposition available in the restricted setting, makes them indispensable for studying Lie-theoretic structures in positive characteristic. So far, a number of restricted algebraic structures, extending the notion of a restricted Lie algebra to other types of algebras, have been introduced. These include restricted Poisson algebras~\cite{BK0,BYZ}, restricted pre-Lie algebras~\cite{Bu,Dz,Do}, restricted Lie-Rinehart algebras~\cite{Do1}, and restricted post-Lie algebras~\cite{EG}.

This paper focuses on restricted versions of two closely related algebraic structures: the post-Lie algebra and the Rota-Baxter Lie algebra. The concept of a post-Lie algebra was originally introduced by Vallette in the context of Koszul duality of operads in \cite{Val}. Beyond its algebraic origin, post-Lie algebra structures can be derived in differential geometry from affine flat connections with constant torsion. They also play significant parts in several applied fields, including numerical integration on manifolds \cite{ML} and the theory of regularity structures for stochastic partial differential equations \cite{BK,JZ}.

As a post-Lie structure over an abelian Lie algebra, a pre-Lie algebra originates from Cayley's 1896 work on rooted tree algebras. In the literature, this algebraic structure is also referred to as a Vinberg algebra or a left-symmetric algebra in the context of convex homogeneous cones, and as a Gerstenhaber algebra or a right-symmetric algebra in the setting of deformation theory. A significant technical development was later introduced by Guin and Oudom in~\cite{OG}, who constructed a certain universal enveloping algebra for a pre-Lie algebra. This construction was subsequently extended to post-Lie algebras and has found deep applications in the Lie-Butcher theory of numerical methods~\cite{ELM,Ra}.

On the other hand, a Rota-Baxter operator of weight 0 on a Lie algebra provides an operator form of the classical Yang-Baxter equation~\cite{STS}. The correct framework for extending classical $r$-matrices via their operator forms was introduced by Bordemann and later termed an $\mathcal{O}$-operator by Kupershmidt~\cite{Bor,Ku}. Bai, Guo, and Ni~\cite{BGN} further generalized this notion by introducing $\mathcal{O}$-operators of weight $\lambda$ (for a constant $\lambda$) and extended $\mathcal{O}$-operators, which notably encompass Rota-Baxter operators of weight $\lambda$ on Lie algebras and have applications to the study of Lie bialgebras. In particular, the Rota-Baxter identity of weight 1 on a Lie algebra is equivalent to the modified classical Yang-Baxter equation.
Importantly, the authors in~\cite{Ag,BGN} pointed out a splitting property: a Rota-Baxter operator, or more generally an $\mathcal{O}$-operator of weight $\lambda$ on a Lie algebra, naturally induces an underlying algebraic structure, either a pre-Lie algebra or {further} a post-Lie algebra.

Although thousands of papers have explored diverse aspects of Rota-Baxter algebras, to the best of our knowledge, very few have examined Rota-Baxter (associative or Lie) algebras in positive characteristic. Any new phenomenon observed beyond the case of characteristic 0 would naturally invite further investigation. In particular, the compatibility of a Rota-Baxter operator with a restricted Lie algebra structure is of paramount importance; we anticipate that many basic properties of Rota-Baxter Lie algebras can be carried over to restricted Rota-Baxter Lie algebras. Moreover, establishing natural constructions for such algebras is essential to validate the coherence and relevance of this extended framework.
{Simultaneously}, the notion of a restricted post-Lie algebra remained mysterious until very recently. In~\cite{EG}, Ehret and Gilliers adapted the Guin-Oudom procedure to post-Lie algebras in characteristic $p$, thereby establishing a prototype for restricted post-Lie algebras. Inspired by their work, we introduce in this paper the notion of a restricted Rota-Baxter Lie algebra of weight $\lambda$, especially unifying the cases of weight 0 and of weight 1. We study its fundamental properties, examine in detail its relationship with restricted post-Lie algebras, and provide auxiliary constructions.

The paper is organized as follows. In Section~\ref{sec:rla-pla}, we review the definitions of restricted Lie algebras, their restricted enveloping algebras, as well as post-Lie algebras and their universal enveloping algebras.

Section~\ref{sec:rpla} begins by recalling the notion of a (trivially) restricted post-Lie algebra introduced by Ehret and Gilliers. We then present two constructions of such algebras: one via the universal $p$-envelopes of post-Lie algebras using the Guin-Oudom extension (Theorem~\ref{thm:pla-rpla}), and another via restricted action post-Lie algebras arising from derivation actions of restricted Lie algebras on commutative algebras (Theorem~\ref{thm:rapla}).

In Section~\ref{sec:rrla}, we first recall Rota-Baxter (Lie) algebras and then introduce the central notion of a restricted Rota-Baxter Lie algebra of arbitrary weight, {which can be intrinsically interpreted by graph subalgebras (Theorem~\ref{thm:rrb-graph})}. A characterization of idempotent restricted Rota-Baxter operators on restricted Lie algebras is provided (Theorem~\ref{thm:idem-rrb}). Furthermore, we extend the novel splitting and replication properties of Rota-Baxter Lie algebras to the restricted setting. Specifically, every restricted Rota-Baxter Lie algebra naturally induces a splitting restricted post-Lie algebra structure (Theorem~\ref{thm:rpl-rrb}) and a descendent restricted Rota-Baxter Lie algebra (Theorem~\ref{thm:des-rrb}). We also establish a Rota-Baxter enhancement of Jacobson's classical theorem on the restrictability of Lie algebras in characteristic $p$ (Theorem~\ref{thm:restrictability}).

To substantiate the concept of a restricted Rota-Baxter Lie algebra, we furnish two natural examples. First, the commutator Rota-Baxter Lie algebra associated to a Rota-Baxter algebra of arbitrary weight in characteristic $p$ carries a restricted structure (Theorem~\ref{thm:rba-rrbl}). Second, the primitive Lie subalgebra of the universal enveloping algebra of a  Rota-Baxter Lie algebra of weight $1$ in characteristic $p$ admits a restricted Rota-Baxter operator (Theorem~\ref{thm:rrbl}), and further serves as a universal Rota-Baxter $p$-envelope of such a Rota-Baxter Lie algebra (Theorem~\ref{thm:urbe}).

{\bf Notation.} Let $\bk$ be an algebraically closed base field of prime characteristic $p$.
All the objects under discussion, including vector spaces, algebras and tensor products, are taken over $\bk$ by default. Also, we denote $[n]\coloneqq \{1,\dots,n\}$ for any positive integer $n$.

\section{Restricted Lie algebras and post-Lie algebras}\label{sec:rla-pla}
This section first recalls the concept of a restricted Lie algebra and its associated restricted enveloping algebra.
\begin{defn}[\cite{Ja}]\label{defn:rLie}
A {\bf restricted} Lie algebra $(\g,[-,-],(-)^{[p]})$ over a field $\bk$
of characteristic $p> 0$ is a Lie algebra $\g=(\g,[-,-])$ over $\bk$ together with a
map $(-)^{[p]}: \g \rightarrow \g$ called the {\bf $p$-map} such that the following relations hold:
\begin{align}
\label{eq:rl-1}
(\alpha x)^{[p]} &=  \alpha^{p}x^{[p]},\\
\label{eq:rl-2}
[x^{[p]},y ]     &=  [\underbrace{x,[x,[\cdots [x}_{p},y]\cdots ]]], \\
\label{eq:rl-3}
(x+y)^{[p]} &= x^{[p]}+y^{[p]}+\sum_{i=1}^{p-1}s_{i}(x,y),
\end{align}
where $x,y \in \g$, $\alpha \in \bk$ and $i s_{i}(x,y)$ is the coefficient of $t^{i-1}$ in the
formal $(p-1)$-fold product $$[\underbrace{t x+y,[\cdots [t x
+y}_{p-1},x]\cdots]]\in\bk[t]\otimes \g.$$
\end{defn}

In \cite{Ja} Jacobson gave the following basic example of restricted Lie algebras.
\begin{exam}\label{ex:ass-rla}
Let $A$ be an associative algebra over a field $\bk$
of characteristic $p>0$. The {\bf commutator Lie algebra} of $A$, denoted by $A^-$ and equipped with the $p$-th power map $(-)^p$, is a restricted Lie algebra. Moreover,
the Lie algebra $\Der(A)$ of derivations on $A$  equipped with the $p$-th power map $(-)^p$ is also a restricted Lie algebra. See also \cite[\S 2.1]{SF}.

Especially, if $A$ is a commutative associative algebra in characteristic $p$. For any $a\in A$ and $D\in \Der(A)$, $aD$ is also a derivation on $A$ and Eq.~(1.4) in \cite{Do1} tells us that
\begin{equation}\label{eq:p-derivation}
(aD)^p=a^pD^p+(aD)^{p-1}(a)D,
\end{equation}
where $a\in A$ serves as the left multiplication of $a$ on $A$. Such a prototype $(A,\Der(A),\id,(-)^p)$ leads to the general notion of a restricted Lie-Rinehart algebra; see~\cite[Definition~1.7]{Do1}.
\end{exam}

For a restricted Lie algebra $(\g,[-,-],(-)^{[p]})$, let $\U(\g)$ be the universal enveloping algebra of the Lie algebra $(\g,[-,-])$. Then in $\U(\g)$, we have
$$[x^p,y]\ =\ x^py-yx^p\ =\ (L_x^p-R_x^p)y \ =\ (L_x-R_x)^py\ = \ (\ad_x)^py 
\ \stackrel{\eqref{eq:rl-2}}{=}\ [x^{[p]},y],\quad x,y\in\g,$$
where $L_x$ and $R_x$ are the left multiplication and the right multiplication of $x$ on $\U(\g)$ respectively. So $x^p-x^{[p]}$ is a central primitive element in $\U(\g)$.

\begin{defn}[\cite{Ja}]
The {\bf restricted enveloping algebra} of a restricted Lie algebra $(\g,[-,-],(-)^{[p]})$ is the quotient algebra
$$\uu(\g)=\U(\g)/(x^p-x^{[p]}\,|\,x\in \g).$$
\end{defn}

\begin{exam}\label{ex:uea-rla}
Let $\g=(\g,[-,-])$ be a Lie algebra over a field $\bk$
of characteristic $p>0$. According to \cite[Proposition~5.5.3]{Mon}, the primitive Lie subalgebra $P(\U(\g))=\{X\in \U(\g)\,|\,\Delta(X)=X\otimes 1+1\otimes X\}$ of its universal enveloping algebra $\U(\g)$ is given by
$$\g^{(p)}\coloneqq {\rm span}_\bk\{x^{p^k}\in \U(\g)\,|\,x\in \g,\ k\geq0\},$$
with its Lie bracket $[-,-]$ coming from the commutator of $\U(\g)$.
Hence, $(\g^{(p)},[-,-],(-)^p)$ is a restricted Lie subalgebra of $(\U(\g)^-,(-)^p)$,
and its restricted enveloping algebra $\uu(\g^{(p)})$ is isomorphic to $\U(\g)$ as Hopf algebras. The tuple $(\g^{(p)},(-)^p,i)$ is also a universal $p$-envelope of the Lie algebra $\g$, where $i:\g\to \g^{(p)}$ is the natural inclusion; see~\cite[Theorem 5.2]{SF}.

\end{exam}

Next we recall the notion of a post-Lie algebra and their universal enveloping algebras.
\begin{defn}\cite{Val}
A {\bf post-Lie algebra} $(\g ,[-,-] ,\rhd)$ consists of a Lie algebra $(\g,[-,-] )$ and a binary product $\rhd:\g \otimes \g \to \g $ such that
\begin{eqnarray}
\label{Post-L-1}x\rhd[y,z] &=&[x\rhd y,z] +[y,x\rhd z] ,\\
\label{Post-L-2}([x,y] +x\rhd y-y\rhd x)\rhd z&=&x\rhd(y\rhd z)-y\rhd(x\rhd z).
\end{eqnarray}
When $(\g ,[-,-] )$ is abelian, $(\g ,[-,-] ,\rhd)$ is a {\bf pre-Lie algebra}.
Any post-Lie algebra $(\g,[-,-] ,\rhd)$ has a {\bf subjacent Lie algebra}
$\g_\rhd\coloneqq (\g,\llb-,-\rrb)$
defined by
\begin{equation}\label{eq:post-L-sub}
\llb x,y\rrb \coloneqq x\rhd y-y\rhd x+[x,y],\quad\forall x,y\in \g.
\end{equation}

\end{defn}

Given a post-Lie algebra $(\g,[-,-],\rhd)$,
let $\U(\g)$ be the universal enveloping algebra of the Lie algebra $(\g,[-,-])$. It has the usual coshuffle Hopf algebra structure $(\U(\g),\cdot,\Delta_\shuffle,S)$.
Recall that the coshuffle coproduct $\Delta_\shuffle$ is defined by $\Delta_\shuffle(x)=x\otimes 1+1\otimes x$ for any $x\in \g$. More generally, the iterated coshuffle coproduct on $\U(\g)$ is given by
\begin{equation}\label{eq:cosh_coproduct}
\Delta_\shuffle^{(r-1)}(x_1\cdots x_m) =
\sum_{\pi} x_{B_1}\otimes\cdots\otimes x_{B_r},\quad \forall x_1,\dots, x_m\in \g,\ r\geq1,
\end{equation}
where the sum ranges over all partitions $\pi$ of $[m]$ into a tuple $(B_1,\dots,B_{r})$ of $r$ possibly empty subsets, and we use the notation $x_I\coloneqq x_{i_1}\cdots x_{i_r}\in \U(\g)$ for any $I=\{i_1<\cdots< i_r\}\subseteq[m]$. In particular, $x_{\emptyset}=1$.
On the other hand, the antipode $S$ is given by
$$S(x_1 x_2 \cdots x_m)=(-1)^m x_m x_{m-1} \cdots x_1,\quad\forall x_1,x_2,\dots,x_m\in \g.$$

By the Guin-Oudom construction~\cite{ELM,OG}, the post-Lie product $\rhd$ on $\g$ can be extended to $\U(\g)$, so that $(\U(\g),\cdot\,,\Delta_\shuffle,S,\rhd)$ is a post-Hopf algebra introduced in \cite[Definition~2.1]{LST}, namely
\begin{eqnarray}
\label{Post-1}x\rhd (y\cdot z)&=&(x_1\rhd y)\cdot(x_2\rhd z),\\
\label{Post-2}x\rhd (y\rhd z)&=&\big(x_1\cdot(x_2\rhd y)\big)\rhd z
\end{eqnarray}
for any $x,y,z\in \U(\g)$, where we use the abbreviated Sweedler notation $\Delta(x)=x_1\otimes x_2$, and the linear map $\alpha_\rhd:\U(\g)\to \End(\U(\g))$ defined by
$$\alpha_{\rhd, x} y= x\rhd y,\quad\forall x,y\in \U(\g),$$
has a convolution inverse $\beta_\rhd$ such that $\alpha_{\rhd,x_1}\beta_{\rhd,x_2}=\beta_{\rhd,x_1}\alpha_{\rhd,x_2}=\vep(x)\id_{\U(\g)}$ for any $x\in \U(\g)$.

Correspondingly, there exists a {\bf subjacent Hopf algebra}
$$\U(\g)_\rhd \coloneqq (\U(\g),*_\rhd,\Delta_\shuffle,S_\rhd),$$
where the generalized Grossman-Larson product $*_\rhd$ and the antipode $S_\rhd$ are defined as follows,
\begin{eqnarray}
\label{post-rbb-1}x *_\rhd  y&\coloneqq&x_1\cdot (x_2\rhd y),\\
\label{post-rbb-2}S_\rhd(x)&\coloneqq&\beta_{\rhd,x_1}(S(x_2))
\end{eqnarray}
for any $x,y\in \U(\g)$.

Moreover, there is a Hopf algebra isomorphism $\phi:\U(\frakg_\rhd)\to \U(\frakg)_\rhd$ defined by $$\phi(x_1\cdots x_r)=x_1 *_\rhd \cdots *_\rhd x_r,\quad \forall x_1,\dots,x_r\in \frakg.$$
We call $\phi$ {\bf the Guin-Oudom isomorphism}.

\medskip
For the universal enveloping algebra $\U(\g)$ of a post-Lie algebra $(\g,[-,-],\rhd)$ over $\bk$, the formula (18) in \cite{EMM} tells us that
$$x^{*_\rhd n}=\sum_{\pi}
x^{\rhd |B_1|}\cdots x^{\rhd |B_r|},\quad x\in \g^{(p)},$$
where the sum ranges over all set partitions $\pi=(B_1,\dots,B_r)$ of $[n]$ with blocks $B_1,\dots, B_r$ satisfying $\max B_1<\cdots<\max B_r$, and
$x^{\rhd k}\coloneqq\overbrace{x \rhd (x \rhd \cdots (x}^{k-1}\rhd x)\cdots)$, $k\geq1$. 
Namely, we have
\begin{equation}\label{eq:gl-expand}
x^{*_\rhd n}=\sum_{\alpha\vDash n}C_\alpha x^{\rhd \alpha},\quad x\in \g^{(p)},
\end{equation}
where the sum ranges over all compositions $\alpha=(\alpha_1,\dots,\alpha_r)$ of $n$, $x^{\rhd \alpha}\coloneqq x^{\rhd \alpha_1}\cdots x^{\rhd \alpha_r}$ and
\begin{equation}\label{eq:set-part}
C_\alpha={\alpha_1-1\choose \alpha_1-1}{\alpha_1+\alpha_2-1\choose \alpha_2-1}\cdots{n-1\choose \alpha_r-1}=\dfrac{(n-1)!}{(\alpha_1-1)!\cdots(\alpha_r-1)!}
\prod_{k=1}^{r-1}(\alpha_1+\cdots+\alpha_k)^{-1}
\end{equation}
counts the number of set partitions $\pi=(B_1,\dots,B_r)$ of $[n]$ with blocks $B_1,\dots,B_r$ satisfying $\max B_1<\cdots<\max B_r$ and $|B_i|=\alpha_i$ for all $i\geq1$.

\medskip
By an analog in characteristic $p$ of the Dynkin-Specht-Wever Theorem~(e.g. \cite[Chapter V, Theorem 8]{Ja}), Eq.~\eqref{eq:gl-expand} can be rewritten as the following useful formula for the $p$-th power $x^{*_\rhd p}$; see also~\cite[Theorem~2.3]{EG}.
\begin{lemma}
For the universal enveloping algebra $\U(\g)$ of a post-Lie algebra $(\g,[-,-],\rhd)$ in characteristic $p>0$, the following equality holds:
\begin{equation}\label{eq:glproduct}
x^{*_\rhd p}=x^p+\sum_{r=1}^{p-1}\dfrac{1}{r}\sum_{\alpha\vDash p\atop\ell(\alpha)=r}C_\alpha[x^{\rhd \alpha}],\quad  x\in \g^{(p)},
\end{equation}
where the sum ranges over all compositions $\alpha=(\alpha_1,\dots,\alpha_r)$ of $p$ with length $r<p$ and
$$[x^{\rhd \alpha}]
\coloneqq [x^{\rhd \alpha_1},[x^{\rhd \alpha_2},[\cdots[x^{\rhd \alpha_{r-1}},x^{\rhd \alpha_r}]\cdots]]].$$
Also, the fraction $\frac{1}{r}$ therein is interpreted as the inverse of $r$ modulo $p$.
\end{lemma}

Also, we need the following lemma.
\begin{lemma}\label{2.7}
For a post-Lie algebra $(\g,[-,-],\rhd)$ over a field $\bk$ of characteristic $p$, the following formula holds for its universal enveloping algebra $\U(\g)$:
\begin{equation}\label{eq:der-uea}
y\rhd x^p=[\underbrace{x,[x,[\cdots [x}_{p-1},y\rhd x]\cdots ]]],\quad x,y\in \g^{(p)}.
\end{equation}
\end{lemma}
\begin{proof}
For any $x,y\in \g^{(p)}$, we check Eq.~\eqref{eq:der-uea} as follows.
\begin{eqnarray*}
[\underbrace{x,[x,[\cdots [x}_{p-1},y\rhd x]\cdots ]]]&=&(L_x-R_x)^{p-1}(y\rhd x)\ =\ \sum_{i=0}^{p-1}{p-1\choose i}(-1)^iL_x^iR_x^{p-1-i}(y\rhd x)\\
&=& \sum_{i=0}^{p-1}L_x^iR_x^{p-1-i}(y\rhd x)\ =\ \sum_{i=0}^{p-1}x^i(y\rhd x)x^{p-1-i}\ \stackrel{\eqref{Post-1}}{=}\ y\rhd x^p,
\end{eqnarray*}
where the left multiplication $L_x$ and the right multiplication $R_x$ commute with each other.
\end{proof}

\section{Restricted post-Lie algebras and their constructions}\label{sec:rpla}

Now we recall the notion of a (trivially) restricted post-Lie algebra introduced by Ehret and Gilliers, and then discuss two kinds of constructions of them.
\begin{defn}[\cite{EG}]\label{defn:rpl}
A {\bf restricted post-Lie algebra} $(\g,[-,-],(-)^{[p]},\rhd)$  over a field $\bk$ of characteristic $p$ is a post-Lie algebra $(\g,[-,-],\rhd)$
such that $(\g,[-,-],(-)^{[p]})$ is a restricted Lie algebra and
\begin{align}
\label{eq:rpl-3}
y\rhd x^{[p]}  &= [\underbrace{x,[x,[\cdots [x}_{p-1},y\rhd x]\cdots ]]],\\
\label{eq:rpl-4}
x^{\llb p\rrb}\rhd y     &=  \underbrace{x \rhd (x \rhd \cdots (x}_{p}\rhd y)\cdots),
\end{align}
where the map $(-)^{\llb p\rrb}:\g\to \g$ is defined as follows referring to the RHS of Eq.~\eqref{eq:glproduct}:
\begin{align}
\label{eq:rpl-2}
x^{\llb p\rrb}  &= x^{[p]} +\sum_{r=1}^{p-1}\dfrac{1}{r}\sum_{\alpha\vDash p\atop\ell(\alpha)=r}C_\alpha[x^{\rhd \alpha}].
\end{align}
Namely, the following equality holds for central elements $x^p-x^{[p]}$ in $\U(\g)$:
\begin{equation}\label{eq:p-map-uea}
x^{*_\rhd p}-x^{\llb p\rrb}=x^p-x^{[p]},\quad x\in \g,
\end{equation}
where $*_\rhd$ is the generalized Grossman-Larson product of $\U(\g)_\rhd$ defined by Eq.~\eqref{post-rbb-1}. Also, Eqs.~\eqref{eq:rpl-3} and \eqref{eq:rpl-4} are equivalent to that the left multiplication $\alpha_{\rhd,x}$ is a restricted derivation such that $\alpha_{\rhd, x^{\llb p\rrb} }=\alpha_{\rhd,x}^p$.
\end{defn}

\begin{remark}
(1) Definition \ref{defn:rpl} of restricted post-Lie algebras is actually that of trivially restricted post-Lie algebra in \cite[Definition~3.1]{EG}. Removing the constraint \eqref{eq:rpl-2} (equivalent to Eq.~(2) in \cite{EG}), there is a general definition of restricted post-Lie algebras therein (\cite[Definition~3.8]{EG}). But we especially concern Eq.~\eqref{eq:rpl-2} in order to posing later our central notion of a restricted Rota-Baxter Lie algebra with naturally inherited properties from classical ones.

(2) For a restricted post-Lie algebra $(\g,[-,-],(-)^{[p]},\rhd)$ in characteristic $p>0$ such that $[-,-]=0$, Eq.~\eqref{eq:rpl-2} reduces to
$$x^{\llb p\rrb}=x^{[p]}+x^{\rhd p},$$
and $(\g,\rhd,(-)^{\llb p\rrb})$ is actually a restricted pre-Lie algebra in the sense of \cite{Do}. 
More specially, if $(-)^{[p]}=0$ such that $x^{\llb p\rrb}=x^{\rhd p}$,
then $(\g,\rhd,(-)^{\rhd p})$ is a restricted pre-Lie algebra originally introduced in \cite{Dz}.

\end{remark}

Next we mention two useful properties of restricted post-Lie algebras.
\begin{prop}[{\cite[Prop.~3.2]{EG}}]\label{prop:sub-ad-restrict}
Given a restricted post-Lie algebra $(\g,[-,-],(-)^{[p]},\rhd)$, it gives rise to a {\bf subjacent restricted Lie algebra} $(\g,\llb-,-\rrb,(-)^{\llb p\rrb})$, where $\llb -,-\rrb$ is the subjacent Lie bracket of the post-Lie algebra $(\g,[-,-],\rhd)$ in Eq.~\eqref{eq:post-L-sub} and $(-)^{\llb p\rrb}$ is defined in Eq.~\eqref{eq:rpl-2}.
\end{prop}
\begin{proof}
The subjacent Lie algebra $\g_\rhd=(\g,\llb-,-\rrb)$ is a Lie subalgebra of the restricted commutator Lie algebra $(U(\g)_\rhd^-,(-)^{*_\rhd p})$, and Eq.~\eqref{eq:p-map-uea} implies that
$f:x\mapsto x^{*_\rhd p}-x^{\llb p\rrb}$ is a $p$-semilinear map from $\g_\rhd$ to the centralizer $C_{U(\g)_\rhd^-}(\g_\rhd)$, so $(-)^{\llb p\rrb}$ is a $p$-map on $\g_\rhd$ by \cite[Prop.~2.1]{SF}.
\end{proof}

\begin{prop}[{\cite[Prop.~3.3]{EG}}]\label{prop:res-env-rpla}
Given a restricted post-Lie algebra $(\g,[-,-],(-)^{[p]},\rhd)$, the post-Hopf product $\rhd$ on $\U(\g)$ can descend to the restricted enveloping algebra $\uu(\g)$, so that $\uu(\g)$ is a post-Hopf quotient algebra of $\U(\g)$.
\end{prop}

\subsection{Universal $p$-envelopes of post-Lie algeras}

\begin{defn}
	Let $(\g,(-)^{[p]},\rhd)$ be a restricted post-Lie algebra over $\bk$.
A post-Lie subalgebra $\lieh$ of $\g$ is a $p$-subalgebra if $x^{[p]}\in \lieh,\ \forall x\in \lieh$. For a subset $E\subset \g$, we denote by $E_p$ the intersection of all $p$-subalgebras containing $E$, which is a $p$-subalgebra of $\g$ and is referred to as the $p$-subalgebra generated by $E$ in $\g$.
\end{defn}	

\begin{defn}
Let $(\g,\rhd)$ be a post-Lie algebra.
A tuple $(\liel, (-)^{[p]'}, \rhd', i)$ consisting of a restricted post-Lie algebra $(\liel, (-)^{[p]'}, \rhd')$ and a post-Lie algebra homomorphism $i:\g\to \liel$ is called a {\bf $p$-envelope} of $(\g,\rhd)$ if $i$ is injective and $i(\g)_p=\liel.$
	
A $p$-envelope $(\liel, (-)^{[p]'}, \rhd', i)$ of $(\g,\rhd)$ is called {\bf universal} if it satisfies the following universal property: for any restricted post-Lie algebra $\liec$ and post-Lie algebra homomorphism $f:\g\to \liec$, there exists a unique restricted post-Lie algebra homomorphism $g: \liel\to \liec$ such that $g\circ i=f$.
\end{defn}

By Example~\ref{ex:uea-rla}, any Lie algebra $\g$ in characteristic $p$ naturally gives rise to a restricted Lie algebra $\g^{(p)}$ inside $\U(\g)$ as a universal $p$-envelope of $\g$. Next we show that such a construction can be strengthen for post-Lie algebras.
\begin{theorem}\label{thm:pla-rpla}
Given a post-Lie algebra $(\g,[-,-],\rhd)$ in characteristic $p$, let $\g^{(p)}$ be the primitive Lie subalgebra of $\U(\g)^-$
as in Example~\ref{ex:uea-rla}. Then $(\g^{(p)},[-,-],(-)^p,\rhd,i)$ is a universal $p$-envelope of $(\g,[-,-],\rhd)$, where $i:\g\to \g^{(p)}$ is the natural embedding.
\end{theorem}
\begin{proof}
First we show that $(\g^{(p)},[-,-],\rhd)$ is a post-Lie algebra.
Since $\g^{(p)}$ is the primitive Lie subalgebra of $\U(\g)^-$,
the coalgebra homomorphism $\rhd:\U(\g)\otimes \U(\g)\to \U(\g)$ has a restriction on $\g^{(p)}$
and Eq.~\eqref{Post-L-1} holds for $\g^{(p)}$ by Eq.~\eqref{Post-1}.
On the other hand, for any $x,y,z\in\g^{(p)}$,
we have
$$(x\rhd y- y\rhd x + [x,y])\rhd z=(xy + x\rhd y)\rhd z - (yx + y\rhd x)\rhd z \stackrel{\eqref{Post-2}}{=} x\rhd(y\rhd z)-y\rhd(x\rhd z).$$
Namely, Eq.~\eqref{Post-L-2} also holds.

Next we have known that $(\g^{(p)},[-,-],(-)^p)$ is a restricted Lie algebra in Example~\ref{ex:uea-rla}. Since
$$\llb x,y\rrb \stackrel{\eqref{eq:post-L-sub}}{=} x\rhd y- y\rhd x + [x,y] \stackrel{\eqref{post-rbb-1}}{=} x*_\rhd y - y*_\rhd x,\quad x,y \in\g^{(p)},$$
$(\g^{(p)},\llb-,-\rrb)$ is the primitive Lie subalgebra of $\U(\g)^-_\rhd$, which is stable under the power map $(-)^{*_\rhd p}$.
According to Example~\ref{ex:ass-rla},
$(\g^{(p)},\llb-,-\rrb,(-)^{*_\rhd p})$ is a restricted Lie subalgebra of the restricted commutator Lie algebra $(\U(\g)_\rhd^-,(-)^{*_\rhd p})$.

Since $x^{[p]}=x^p$ for all $x\in\g^{(p)}\subset U(\g),$  condition~\eqref{eq:rpl-3} holds for $\g^{(p)}$ according to Lemma \ref{2.7}, and $ x^{\llb p \rrb}= x^{*_\rhd p} $
by the definition of $(-)^{\llb p \rrb}$ and Eqs. \eqref{eq:glproduct}.
Moreover, condition~\eqref{eq:rpl-4} is due to Eqs.~\eqref{Post-2} and \eqref{post-rbb-1}, as $(\U(\g),\rhd)$ is a post-Hopf algebra.
Hence, $(\g^{(p)},[-,-],(-)^p,\rhd)$ is a restricted post-Lie algebra. Since $\g^{(p)}$ is the $p$-subalgebra of $\U(\g)^-$ generated by $\g$ and with the natural embedding $i:\g\to \g^{(p)}$, it is clear that $(\g^{(p)},(-)^p,\rhd,i)$ is a $p$-envelope of $(\g,\rhd)$.
	
Now, for each restricted post-Lie algebra $(\liel,(-)^{[p]'},\rhd')$ and a post-Lie algebra homomorphism $f:\g \to \liel$, we can apply Proposition~\ref{prop:res-env-rpla} to $(\liel,(-)^{[p]'},\rhd')$ and use the universal property of $\U(\g)$ to obtain a post-Hopf algebra homomorphism $\bar{f}: \U(\g)\to \uu(\liel)$ such that $\bar f\circ i=f$.
Now for any $x\in\g$ and $k\geq0$, we have
$$\bar{f}(x^{p^k})=\bar{f}(x)^{p^k}=f(x)^{p^k}=f(x)^{{[p]'}^k}\in\liel,$$
so $\g^{(p)}\subseteq \bar{f}^{-1}(\liel)$ and $\bar{f}(X^p)=\bar{f}(X)^{[p]'}$ for any $X\in \g^{(p)}$. Then $g\coloneqq \bar{f}|_{\g^{(p)}}: \g^{(p)} \to \liel$ is a restricted post-Lie algebra homomorphism such that $g\circ i=f$. Since $(\g^{(p)},(-)^p,i)$ is the $p$-subalgebra generated by $\g$, such a homomorphism $g$ is unique and $(\g^{(p)},(-)^p,\rhd,i)$ is a universal $p$-envelope of $(\g,\rhd)$.

\end{proof}

Given a magma algebra $(V,\rhd)$, let $T(V)$ be the tensor algebra over $V$ and let ${\rm Lie}(V)$ be the Lie subalgebra of $T(V)^-$ generated by $V$.
According to Foissy's construction in \cite{Foissy}, there exists an extended product $\rhd$ on the tensor algebra $T(V)$ such that $(T(V),\cdot\,,\Delta_\shuffle,S,\rhd)$
is the universal enveloping algebra of the free post-Lie algebra $({\rm Lie}(V),\rhd)$ over the magma algebra $(V,\rhd)$.
Hence, applying Theorem~\ref{thm:pla-rpla} to the post-Lie algebra $({\rm Lie}(V),\rhd)$, we have
\begin{prop}
Given a magma algebra $(V,\rhd)$ over a field $\bk$ of characteristic $p$,
let ${\rm Lie}(V)^{(p)}$ be the subspace of $T(V)$ spanned by all $X^{p^k}$ for $X\in {\rm Lie}(V)$ and $k\geq0$. Then the quadruple
$$({\rm Lie}(V)^{(p)},[-,-],(-)^p,\rhd)$$
is the free restricted post-Lie algebra over $(V,\rhd)$.
\end{prop}

\delete{
\begin{lemma}
	If $(\lieg, [p], \rhd)$ is a restricted post-Lie algebra, then the center $(C(\lieg), [p], \rhd)$ is a restricted post-Lie subalgebra of $\lieg$ and $\lieg\rhd C(\lieg) \subseteq C(\lieg).$ Moreover, if $ C(\lieg) \rhd\lieg\subseteq C(\lieg),$ then the adjoint Lie algebra $(\on{ad} \lieg, [p]_1, \rhd_1)$ possesses a restricted post-Lie algebra structure where
	\[ (\on{ad}x)^{[p]_1}\coloneqq\on{ad}(x^{[p]})=\underbrace{\on{ad}x \circ \cdots \circ \on{ad}x}_{p},\text{ and } \on{ad}x \rhd_1 \on{ad}y\coloneqq \on{ad} (x\rhd y),  \]
	for all $x,y\in \lieg.$
\end{lemma}
\begin{proof}
	It is clear that $(C(\lieg), [p])$ is a restricted Lie subalgebra of $(\lieg, [p]).$
	
	For every $x\in C(\lieg), y,z\in \lieg, $ since $[y\rhd x,z]=y\rhd [x,z]- [x, y\rhd z]=0,$ we have that $y\rhd x\in C(\lieg)$ and hence the first statement holds.
	
	Since $(\lieg,[p])$ is restricted, one can check that $(\on{ad} \lieg, [p]_1)$ is also a restricted Lie algebra.
	
	The definition of $ \rhd_1 $ on $\ad \lieg$ is well-defined, as established by $\lieg\rhd C(\lieg) \subseteq C(\lieg)$ and $ C(\lieg) \rhd\lieg\subseteq C(\lieg)$. For each $x,y,z\in\lieg,$ since $\on{ad} [x,y]=[\on{ad} x,\on{ad} y],$ we have that
	\begin{eqnarray*}
		\on{ad} x\rhd_1 [\on{ad} y,\on{ad} z]&=&\on{ad} (x\rhd [y,z])=\on{ad} ([x\rhd y,z]+[y,x\rhd z])\\
		&=&[\on{ad} x\rhd_1 \on{ad} y,\on{ad} z]+[\on{ad} y,\on{ad} x\rhd_1\on{ad} z],\\
		{[\on{ad} x,\on{ad} y]\rhd_1\on{ad} z}&=& \on{ad} ([x,y]\rhd z)\\
		&=& \on{ad} \left( x\rhd(y\rhd z)-(x\rhd y)\rhd z- y\rhd(x\rhd z) + (y\rhd x)\rhd z \right)\\
		& =& \on{ad} x\rhd(\on{ad} y\rhd \on{ad} z)-(\on{ad} x\rhd \on{ad} y)\rhd \on{ad} z\\
		&&- \on{ad} y\rhd(\on{ad} x\rhd \on{ad} z) + (\on{ad} y\rhd \on{ad} x)\rhd \on{ad} z.
	\end{eqnarray*}
	Therefore, $(\on{ad} \lieg ,\rhd_1)$ is a post-Lie algebra.
	
	Moreover,
	\begin{eqnarray*}
		\on{ad} x\rhd_1 \on{ad} y^{[p]_1}&=&\on{ad} x\rhd_1 (ad y)^p = \on{ad} x\rhd_1 (ad y^{[p]}) =\on{ad} (x\rhd y^{[p]})  \\
		&=&  \on{ad}  \underbrace{[y ,[ \dots, [y}_{p-1},x\rhd y]\dots ] =  \underbrace{[\on{ad} y ,[ \dots, [\on{ad} y}_{p-1},\on{ad} x\rhd_1 \on{ad} y]\dots ]\\
		(\on{ad} x)^{\llb p\rrb}\rhd_1 \on{ad} y &=& \left((\on{ad} x)^{[p]} + \sum_{r=1}^{p-1}\frac{1}{r} \sum_{\alpha}C_{\alpha}[(\on{ad} x)^{\rhd_1\alpha}]\right)\rhd_1\on{ad} y \\
		&= &\on{ad} (x^{[p]}) \rhd_1\on{ad} y + \sum_{r=1}^{p-1}\frac{1}{r}\sum_{\alpha}C_{\alpha}[(\on{ad} x)^{\rhd_1\alpha}]\rhd_1\on{ad} y \\
		&= & \on{ad}(x^{[p]}\rhd y) + \sum_{r=1}^{p-1}\frac{1}{r}\sum_{\alpha}C_{\alpha}ad[x^{\rhd\alpha}]\rhd_1\on{ad} y \\
		&= & ad(x^{[p]})\rhd y + \sum_{r=1}^{p-1}\frac{1}{r}\sum_{\alpha}C_{\alpha}[x^{\rhd\alpha}]\rhd y)\\
		& =& \on{ad} y(x^{[[p]]}\rhd y) =\on{ad}(\alpha_{\rhd,x}^{p}\rhd y) \\
		&=&  \alpha_{\rhd_1,\on{ad} x}^{p}\rhd_1 \on{ad} y.
	\end{eqnarray*}
	Therefore, $(\on{ad} \lieg, [p]_1,\rhd_1)$ is a restricted post-Lie algebra.
\end{proof}}

\subsection{Restricted action post-Lie algebras}
In this subsection, we give another useful construction of restricted post-Lie algeras, namely the restricted version of action post-Lie algebras originally inspired by \cite[Prop. 4.23]{GLS}. See also \cite[Props.~3.14, 3.19]{EG} for the special situations in characteristic $p = 2, 3$.
\begin{theorem}\label{thm:rapla}
Given a restricted Lie algebra $(\g,[-,-],(-)^{[p]})$ and a commutative associative algebra $A$ over $\bk$, let $\rho:\g\to \Der(A)$ be a restricted Lie algebra homomorphism. Then it induces a restricted post-Lie algebra $(A\otimes \g,[-,-],(-)^{[p]},\rhd)$ as follows.

For any $a,b\in A$ and $x,y \in \g$,
\begin{eqnarray*}
[a\otimes x,b\otimes y] &\coloneqq& ab \otimes [x,y],\\
(a\otimes x)\rhd(b\otimes y) &\coloneqq& ax(b)\otimes y,
\end{eqnarray*}
where we write $\rho(x)(b)$ as $x(b)$ for short, and the $p$-map $(-)^{[p]}$ on $A\otimes \g$ is defined by
\begin{eqnarray*}
(a\otimes x)^{[p]} &\coloneqq& a^p\otimes x^{[p]},
\end{eqnarray*}
and then by Eq.~\eqref{eq:rl-3} for any finite sum of pure tensors.
\end{theorem}
\begin{proof}
It is a standard construction on the intuition of \cite{BK,JZ} that $(A\otimes \g,[-,-],\rhd)$ is a post-Lie algebra. On the other hand, it is easy to check that $(A\otimes \g,[-,-],(-)^{[p]})$ is a restricted Lie algebra over $\bk$, by viewing $(A\otimes \g,[-,-])$ as a Lie $A$-algebra.

Next we check that Eq.~\eqref{eq:rpl-3} holds. Indeed,
$$
(b\otimes y)\rhd (a\otimes x)^{[p]} \ =\ (b\otimes y)\rhd (a^p\otimes x^{[p]}) \ =\ b y(a^p)\otimes x^{[p]}
\ =\ pba^{p-1}y(a)\otimes x^{[p]}=0.
$$
On the other hand,
\begin{eqnarray*}
&&\hspace{-3cm}{[\underbrace{a\otimes x,[a\otimes x,[\cdots [a\otimes x}_{p-1},(b\otimes y)\rhd (a\otimes x)]\cdots ]]]}\\
&=&{[\underbrace{a\otimes x,[a\otimes x,[\cdots [a\otimes x}_{p-1},by(a)\otimes x]\cdots ]]]}\\
&=& ba^{p-1}y(a)\otimes [\underbrace{x,[x,[\cdots [x}_{p-1},x]\cdots ]]]\ =\ 0.
\end{eqnarray*}

Finally, we check that Eq.~\eqref{eq:rpl-4} holds. First we compute that
$$(a\otimes x)^{\llb p\rrb}\ \stackrel{\eqref{eq:rpl-2}}{=}\ (a\otimes x)^{[p]}+(a\otimes x)^{\rhd p}
\ =\ a^p\otimes x^{[p]} +
(a\rho(x))^{p-1}(a)\otimes x,$$
as $[(a\otimes x)^{\rhd \alpha}]=0$ whenever the composition $\alpha$ is of length $\geq2$. Then
\begin{eqnarray*}
 (a\otimes x)^{\llb p\rrb}\rhd(b\otimes y) &=&
 \big(a^p\otimes x^{[p]} +
(a\rho(x))^{p-1}(a)\otimes x\big)\rhd(b\otimes y)\\
 &=& a^p\rho(x^{[p]})(b)\otimes y +
(a\rho(x))^{p-1}(a)\rho(x)(b)\otimes y\\
&=& \big(a^p\rho(x)^p + (a\rho(x))^{p-1}(a)\rho(x)\big)(b)\otimes y,
\end{eqnarray*}
where the last equality holds as $\rho$ is a restricted Lie algebra homomorphism. Also, we have
\begin{eqnarray*}
\underbrace{(a\otimes x) \rhd ((a\otimes x) \rhd \cdots ((a\otimes x)}_{p}\rhd (b\otimes y))\cdots)&=&
(a\rho(x))^p(b)\otimes y.
\end{eqnarray*}
Hence, they are equal by Eq.~\eqref{eq:p-derivation}.
\end{proof}

\section{Restricted Rota-Baxter Lie algebras and their constructions}\label{sec:rrla}

In this section, we introduce restricted Rota-Baxter Lie algebras of arbitrary weight as our primary research subject, examine their basic properties, and provide two natural constructions. First, we recall some necessary concepts referred to \cite{G,BGN}.

\subsection{Rota-Baxter algebras and  Rota-Baxter Lie algebras}

\begin{defn}
Let $\lambda$ be a fixed element in $\bk$. A {\bf Rota-Baxter algebra of weight $\lambda$} is a pair $(A,P)$ where $A$ is
an associative algebra and $P$ is a linear endomorphism of $A$ satisfying
$$P(x)P(y)=P(P(x)y)+P(xP(y))+\lambda P(xy),\quad x,y\in A.$$
The map $P$ is called a {\bf Rota-Baxter operator of weight $\lambda$}.

Note that $(A,P)$ has the {\bf double structure} $A_P\coloneqq (A,\star_P)$, where
$\star_P$ is an associative product on $A$ called the {\bf double product} and defined by
\begin{equation}
\label{eq:RBA-double}
x\star_P y \coloneqq P(x) y+x P(y)+\lambda xy,\quad x,y\in A,
\end{equation}
such that $P:A_P\to A$ is an algebra homomorphism .

Similarly, a {\bf Rota-Baxter Lie algebra $(\g,[-,-],R)$ of weight $\lambda$} has the compatibility
$$[R(x),R(y)]=R([R(x),y])+R([x,R(y)])+\lambda R([x,y]),\quad x,y\in \g.$$
Any Rota-Baxter Lie algebra $(\g,[-,-],R)$ has a {\bf descendent Lie algebra}
$\g_R\coloneqq (\g,[-,-]_R)$ defined by
\begin{equation}\label{eq:RB-L-des}
[x,y]_R \coloneqq [R(x),y]+[x,R(y)]+\lambda[x,y],\quad x,y\in \g,
\end{equation}
such that $R:\g_R\to\g$ is a Lie algebra homomorphism.
\end{defn}

Rota-Baxter Lie algebras have the following splitting property.
\begin{prop}[{\cite[Corollary~5.6]{BGN}}]\label{prop:split}
Let $(\g,[-,-],R)$ be a Rota-Baxter Lie algebra of weight $\lambda$.
Then $(\g,[-,-]_\lambda,\rhd_R)$ is a post-Lie algebra, called the {\bf splitting post-Lie algebra}, where
$$[x,y]_\lambda\coloneqq \lambda[x,y],\quad x\rhd_R y\coloneqq [R(x),y],\quad x,y\in \g.$$
\end{prop}

Also, we need the following result for idempotent Rota-Baxter Lie algebras, analogous to \cite[Prop.~1.1.15]{G}.
\begin{prop}\label{prop:idempotent-rb}
Let $(\g,[-,-],R)$ be a Rota-Baxter Lie algebra of weight $\lambda$. If $R$ is idempotent, then
$$(1+\lambda)R([R(x),y])=(1+\lambda)([R(x),R(y)]-\lambda R([x,y]))=0,\quad x,y\in\g.$$
In particular, such an idempotent Rota-Baxter operator $R$ is a Lie algebra homomorphism from $(\g,[-,-]_\lambda) $ to $(\g,[-,-])$ with $R(\g)$ abelian and $\ker R$ stable under the adjoint action of $R(\g)$, when $\lambda\neq -1$.
\end{prop}
\begin{proof}
Since $R^2=R$ and $R(\g)$ is a Lie subalgebra of $\g$, we have
\begin{align*}
[R(x),R(y)]&=[R^2(x),R(y)]=R([R^2(x),y])+R([R(x),R(y)])+\lambda R([R(x),y])\\
&=R([R(x),R(y)])+(1+\lambda) R([R(x),y])
=[R(x),R(y)]+(1+\lambda) R([R(x),y]).
\end{align*}
Hence, $(1+\lambda) R([R(x),y])=0$ and then
$[R(x),R(y)]-\lambda R([x,y])=R([R(x),y])-R([R(y),x])$ implies that
$(1+\lambda)([R(x),R(y)]-\lambda R([x,y]))=0$.

In particular, if $\lambda\neq-1$, then $R([R(x),y])=0$ and $[R(x),R(y)]= \lambda R([x,y])=R([x,y]_{\lambda})$ for any $x,y\in \g$, and so $[R(x),R(y)]=[R^2(x),R(y)]=\yn{\lambda}R([R(x),y])=0$.
\end{proof}

\subsection{Restricted Rota-Baxter Lie algebras and their properties}
In the light of Definition~\ref{defn:rpl} of restricted post-Lie algebras, we introduce our key notion of a restricted Rota-Baxter Lie algebra of arbitrary weight.
\begin{defn}\label{defi:rrbl}
Let $(\g,[-,-],(-)^{[p]})$ be a restricted Lie algebra over $\bk$. A Rota-Baxter operator $R$ of weight $\lambda\in \bk$ on $\g$ is called {\bf restricted}, if
\begin{equation}\label{eq:rrbl}
R(x^{[p]_R})=R(x)^{[p]},\quad x\in\g,
\end{equation}
where the map $(-)^{[p]_R}:\g\to \g$ is given by
\begin{align}
\label{eq:rrbl-jac}
x^{[p]_R}  &\coloneqq \lambda^{p-1}x^{[p]}+\sum_{r=1}^{p-1}\dfrac{1}{r}\lambda^{r-1}\sum_{\alpha\vDash p\atop\ell(\alpha)=r}C_\alpha[x^{\rhd_R \alpha}]
\end{align}
with respect to the induced product $\rhd_R=[R(-),-]$, where $C_\alpha$ is the coefficient defined in \eqref{eq:set-part}. 
\end{defn}

\begin{remark}
Let $(\g,[-,-],(-)^{[p]},R)$ be a restricted Rota-Baxter Lie algebra of nonzero weight $\lambda$. By definition, $(\g,[-,-],(-)^{[p]},\lambda^{-1}R)$ is a restricted Rota-Baxter Lie algebra of weight $1$.
\end{remark}

First we provide an intrinsic characterization of restricted Rota-Baxter Lie algebras in the perspective of graph subalgebras, which is consistent with the combinatorial formulation in Definition~\ref{defi:rrbl}. A technical lemma is needed.

\begin{lemma}\label{lemma:adjoint_action}
	Given a Lie algebra $\g$ with $D\in \on{Der}(\g)$, $\lambda\in \bk $ and $\ x\in\g, $ one has that
	\begin{equation}\label{eq:adjoint_action}
(\lambda\on{ad}x+D)^n(x) =\sum_{i=1}^{n} \lambda^{i-1} \sum_{\alpha\vDash n+1\atop\ell(\alpha)=i}C_\alpha D^{\alpha}(x),\quad n= 1,\dots,p-1,
	\end{equation}
	where $D^{\alpha}(x)\coloneqq [D^{\alpha_1-1}(x),[ \cdots,[D^{\alpha_{i-1}-1}(x),D^{\alpha_i-1}(x)]\cdots]]$ for $\alpha=(\alpha_1,\dots, \alpha_i)\vDash n+1$.
\end{lemma}
\begin{proof}
Eq.~\eqref{eq:adjoint_action} can be proven by induction on $n$. It is trivial if $n=1.$ Assume that Eq.~\eqref{eq:adjoint_action} holds for some $n\geq1$. Denote $\epsilon_j =(0,\dots,0,\stackrel{j\ {\rm th}}{1},0,\dots)$ for any $j\geq1$. Then, we have
\begin{eqnarray*}
(\lambda\on{ad}x+D)^{n+1}(x)&=& (\lambda\on{ad}x+D)\sum_{i=1}^{n} \lambda^{i-1} \sum_{\beta\vDash n+1\atop\ell(\beta)=i}C_\beta D^{\beta}(x)\nonumber\\
&=& \sum_{i=1}^{n} \lambda^{i} \sum_{\beta\vDash n+1\atop\ell(\beta)=i}C_\beta \on{ad}x\circ D^{\beta}(x)+\sum_{i=1}^{n} \lambda^{i-1} \sum_{\beta\vDash n+1\atop\ell(\beta)=i}C_\beta D\circ D^{\beta}(x)\\
&=& \sum_{i=1}^{n} \lambda^{i} \sum_{\beta\vDash n+1\atop\ell(\beta)=i}C_\beta  D^{(1,\,\beta)}(x)+\sum_{i=1}^{n} \lambda^{i-1} \sum_{\beta\vDash n+1\atop\ell(\beta)=i}C_\beta \sum_{j=1}^{i}D^{\beta+\epsilon_j}(x)\\
&=& \sum_{k=1}^{n+1} \lambda^{k-1} \sum_{\alpha\vDash n+2\atop \ell(\alpha)=k}
C_\alpha D^\alpha(x),
\end{eqnarray*}
where the first equality is due to the induction hypothesis, and the last equality is obtained by the combinatorial meaning
of the coefficients $C_\alpha$'s given in Eq.~\eqref{eq:set-part} as follows.
For any composition $\alpha=(\alpha_1,\dots, \alpha_k)$ of $n+2$, $C_\alpha$ counts the number of set partitions $\pi=(B_1,\dots,B_k)$ of $[n+2]$ with blocks $B_j$ of size $\alpha_j$ and satisfying $\max B_1<\cdots<\max B_k$. Then, given any such $\pi$,
there exists a unique set partition $\pi'$ of $[n+1]$ to obtain $\pi$ by inserting the number $0$ into a block of $\pi'$ or creating one more block $\{0\}$, and then adding $1$ to every number. So the proof is finished.
\end{proof}

Now given a restricted Lie algebras $(\g,[-,-],(-)^{[p]})$ and $\lambda\in\bk$,  $$\g_\lambda\coloneqq(\g,[-,-]_\lambda,(-)^{[p]_\lambda})$$
is also restricted Lie algebra, where $(-)^{[p]_\lambda}\coloneqq\lambda^{p-1}(-)^{[p]}$.
Then Eq.~\eqref{eq:rl-2} tells us that the adjoint map
$\ad:\g\to \Der(\g_\lambda)$
is a restricted Lie algebra homomorphism such that $\ad_x$ is a restricted derivation on $\g_\lambda$ for any $x\in\g$, namely $\ad_{x^{[p]}}=(\ad_x)^p$ and
$$\ad_x(y^{[p]_\lambda})=[\underbrace{y,\cdots [y}_{p-1},\ad_x y]_\lambda\cdots]_\lambda,\quad x\in\g,\ y\in\g_\lambda.$$
By \cite[Theorem~2.5]{SF},
we can define a $p$-map $(-)^{[p]_\rtimes}$ on the semidirect product $\g_\lambda\rtimes_{\ad} \g$ by
$$(x,y)^{[p]_\rtimes}\coloneqq \bigg(x^{[p]_\lambda}+\sum_{i=1}^{p-1}\lambda^{i-1}s_{i}(x,y),y^{[p]}\bigg),$$
where $x,y \in \g$ and $s_i(x,y)$ is given as in Eq.~\eqref{eq:rl-3}, namely
$$[\underbrace{[\lambda x+y,[\cdots [\lambda x+y}_{p-1},x]\cdots]]=\sum_{i=1}^{p-1}\lambda^{i-1}is_{i}(x,y).$$

\begin{theorem}\label{thm:rrb-graph}
Given a restricted Lie algebra $(\g,[-,-],(-)^{[p]})$, a linear operator $R$ on $\g$ is a restricted Rota-Baxter operator of weight $\lambda$ if and only if the corresponding graph subspace
$$G(\g,R)\coloneqq\{(x,R(x))\,|\,x\in \g\}$$
of $\g\oplus\g$ is a restricted Lie subalgebra of $(\g_\lambda\rtimes_{\ad} \g,(-)^{[p]_\rtimes})$.
\end{theorem}
\begin{proof}
Since the Lie bracket $[-,-]_\rtimes$ of $\g_\lambda\rtimes_{\ad} \g$ means that
$$[(x,R(x)),(y,R(y))]_\rtimes=([x,y]_\lambda+[R(x),y]-[R(y),x],[R(x),R(y)]),\quad x,y\in\g,$$
the linear operator $R$ on $\g$ is a Rota-Baxter operator of weight $\lambda$ if and only if  the graph subspace $G(\g,R)$
of $\g\oplus\g$ is a Lie subalgebra of $\g_\lambda\rtimes_{\ad} \g$.

On the other hand, we take $D=\ad_{R(x)}$ and $n=p-1$ in Lemma~\ref{lemma:adjoint_action} to obtain that
$$[\underbrace{[\lambda x+R(x),[\cdots [\lambda x+R(x)}_{p-1},x]\cdots]]=\sum_{i=1}^{p-1}\lambda^{i-1}is_{i}(x,R(x)) \stackrel{\eqref{eq:adjoint_action}}{=} \sum_{i=1}^{p-1}\lambda^{i-1}\sum_{\alpha\vDash p\atop\ell(\alpha)=i}C_\alpha[x^{\rhd_R \alpha}].$$
Hence, comparing the coefficients of $\lambda^{i-1}$ on both sides, we should have
$$x^{[p]_R}
\stackrel{\eqref{eq:rrbl-jac}}{=} \lambda^{p-1}x^{p}+\sum_{i=1}^{p-1}\dfrac{1}{i}\lambda^{i-1}\sum_{\alpha\vDash p\atop\ell(\alpha)=i}C_\alpha[x^{\rhd_R \alpha}]
=x^{[p]_\lambda}+\sum_{i=1}^{p-1}\lambda^{i-1}s_{i}(x,R(x)).$$
That is, $(x,R(x))^{[p]_\rtimes}=(x^{[p]_R},R(x)^{[p]})$.
So $R$ is restricted, i.e. Eq.~\eqref{eq:rrbl} holds, if and only if $G(\g,R)$ is a restricted Lie subalgebra of $(\g_\lambda\rtimes_{\ad} \g,(-)^{[p]_\rtimes})$.
\end{proof}

Next we give a description of idempotent restricted Rota-Baxter Lie algebras. Actually, idempotent Rota-Baxter operators on associative algebras are quite useful in renormalization of quantum field theory; see e.g.~\cite{EGK}.
\begin{theorem}\label{thm:idem-rrb}
Given a restricted Lie algebra $(\g,[-,-],(-)^{[p]})$, an idempotent Rota-Baxter operator $R$ of weight $\lambda\neq -1$ on $\g$ is restricted  if and only if
$R$ is a restricted Lie algebra homomorphism from $(\g,[-,-]_\lambda,(-)^{[p]_\lambda})$ to  $(\g,[-,-],(-)^{[p]})$.
\end{theorem}
\begin{proof}
According to Proposition~\ref{prop:idempotent-rb},
an idempotent Rota-Baxter operator $R$ of weight $\lambda\neq -1$ on $\g$ is a Lie algebra homomorphism from $(\g,[-,-]_\lambda)$ to  $(\g,[-,-])$ such that $R(\g)$ is abelian and $\ker R$ is stable under the adjoint action of $R(\g)$. Therefore, $R(x^{\rhd_R k})=0$ for $k\geq2$ and
\begin{eqnarray*}
R(x^{[p]_R}) &\stackrel{\eqref{eq:rrbl-jac} }{=}& \lambda^{p-1}R(x^{[p]})+\sum_{r=1}^{p-1}\dfrac{1}{r}\lambda^{r-1}
\sum_{\alpha\vDash p\atop\ell(\alpha)=r}C_\alpha R([x^{\rhd_R \alpha_1},[\cdots[x^{\rhd_R \alpha_{r-1}},x^{\rhd_R \alpha_r}]\cdots]])\\
&=& R(x^{[p]_\lambda})+\sum_{r=1}^{p-1}\dfrac{1}{r}
\sum_{\alpha\vDash p\atop\ell(\alpha)=r}C_\alpha
R([x^{\rhd_R \alpha_1},[\cdots[x^{\rhd_R \alpha_{r-1}},x^{\rhd_R \alpha_r}]_\lambda\cdots]_\lambda]_\lambda)\\
&=& R(x^{[p]_\lambda})+\sum_{r=1}^{p-1}\dfrac{1}{r}
\sum_{\alpha\vDash p\atop\ell(\alpha)=r}C_\alpha
[R(x^{\rhd_R \alpha_1}),[\cdots[R(x^{\rhd_R \alpha_{r-1}}),R(x^{\rhd_R \alpha_r})]_\lambda\cdots]_\lambda]_\lambda\ =\ R(x^{[p]_\lambda}).
\end{eqnarray*}
So Eq.~\eqref{eq:rrbl} is equivalent to saying that $R(x^{[p]_\lambda})=R(x)^{[p]}$ for any $x\in\g$, and $R$ is a restricted Lie algebra homomorphism from $(\g,[-,-]_\lambda,(-)^{[p]_\lambda})$ to  $(\g,[-,-],(-)^{[p]})$.
\end{proof}

Rota-Baxter Lie algebras possess the desirable splitting property~\cite[Corollary~5.6]{BGN} and replication property~\cite[Theorem~1.1.17]{G}, which we subsequently extend to the restricted situation.

\begin{theorem}\label{thm:rpl-rrb}
Let $(\g,[-,-],(-)^{[p]},R)$ be a restricted Rota-Baxter Lie algebra of weight $\lambda$. The tuple
$$(\g,[-,-]_\lambda,(-)^{[p]_\lambda},\rhd_R)$$
is a restricted post-Lie algebra, whose subjacent restricted Lie algebra is the descendent restricted Lie algebra $(\g,[-,-]_R,(-)^{[p]_R})$. One calls $\g_\lambda\coloneqq(\g,[-,-]_\lambda,(-)^{[p]_\lambda}, \rhd_R)$ a {\bf splitting restricted post-Lie algebra.}  In particular, for the universal envelope $\U(\g_\lambda)$, we apply Eq.~\eqref{eq:p-map-uea} to obtain that
\begin{equation}\label{eq:p-map-split-uea}
x^{*_{\rhd_R} p}-x^{[p]_R} = x^p-x^{[p]_\lambda},\quad x\in \g_\lambda.
\end{equation}
\end{theorem}
\begin{proof}
We have seen that $(\g,[-,-]_\lambda,(-)^{[p]_\lambda})$ is a restricted Lie algebra, and Proposition~\ref{prop:split} tells us that
$(\g,[-,-]_\lambda,\rhd_R)$ is a post-Lie algebra.
Next we check that Eq.~\eqref{eq:rpl-3} holds.
\begin{eqnarray*}
y\rhd_R x^{[p]_\lambda} &=& \lambda^{p-1}[R(y),x^{[p]}] \ \stackrel{\eqref{eq:rl-2}}{=}\
-\lambda^{p-1}[\underbrace{x,[x,[\cdots [x}_{p},R(y)]\cdots]]]\\
&=& [\underbrace{x,[x,[\cdots [x}_{p-1},y\rhd_R x]_\lambda\cdots]_\lambda]_\lambda]_\lambda.
\end{eqnarray*}
Also, the $p$-map defined by Eq.~\eqref{eq:rpl-2} for $(\g,[-,-]_\lambda,(-)^{[p]_\lambda},\rhd_R)$ is exactly $(-)^{[p]_R}$ in Eq.~\eqref{eq:rrbl-jac}.
Then we see that Eq.~\eqref{eq:rpl-4} holds as follows.
\begin{eqnarray*}
x^{[p]_R}\rhd_R y  &=& [R(x^{[p]_R}),y]
\ \stackrel{\eqref{eq:rrbl}}{=} \ [R(x)^{[p]},y] \\
&\stackrel{\eqref{eq:rl-2}}{=}& [\underbrace{R(x),[R(x),[\cdots [R(x)}_{p},y]\cdots]]]\\
&=& \underbrace{x \rhd_R (x \rhd_R \cdots (x}_{p}\rhd_R y)\cdots).
\end{eqnarray*}
Hence, $(\g,[-,-]_\lambda,(-)^{[p]_\lambda},\rhd_R)$ is a restricted post-Lie algebra.

According to Proposition~\ref{prop:sub-ad-restrict}, $(\g,[-,-]_\lambda,(-)^{[p]_\lambda},\rhd_R)$ has a subjacent restricted Lie algebra, which is exactly $(\g,[-,-]_R,(-)^{[p]_R})$ by Eq.~\eqref{eq:RB-L-des}.
\end{proof}

\begin{theorem}\label{thm:des-rrb}
If $(\g,[-,-],(-)^{[p]},R)$ is a restricted Rota-Baxter Lie algebra of weight $\lambda$, then $(\g_R,(-)^{[p]_R},R)$ is again a restricted Rota-Baxter Lie algebra of weight $\lambda$, and we call it the {\bf descendent restricted Rota-Baxter Lie algebra}. Hence, the Rota-Baxter operator $R$ is a restricted Rota-Baxter Lie algebra homomorphism from the descendent restricted Rota-Baxter Lie algebra $(\g_R,(-)^{[p]_R},R)$ to the original one $(\g,[-,-],(-)^{[p]},R)$.
\end{theorem}
\begin{proof}
The pair $(\g_R,R)$ is a descendent Rota-Baxter Lie algebra of weight $\lambda$, so with the corresponding splitting post-Lie product $\splie_R\coloneqq [R(-),-]_R$ by Proposition~\ref{prop:split}. Moreover, we have
\begin{eqnarray*}
R(x^{\splie_R k}) &=& R\Big([\underbrace{R(x),[R(x),[\cdots [R(x)}_{k-1},x]_R\cdots]_R]_R]_R\Big)\\
&=& [\underbrace{R^2(x),[R^2(x),[\cdots [R^2(x)}_{k-1},R(x)]\cdots]]]\\
&=& R(x)^{\rhd_R k}
\end{eqnarray*}
for any $k\geq 2$.

By Theorem~\ref{thm:rpl-rrb}, we know that $(\g,[-,-]_\lambda,(-)^{[p]_\lambda},\rhd_R)$
is a restricted post-Lie algebra and $(\g_R,(-)^{[p]_R})$ is its subjacent restricted Lie algebra. Also, we check that Eq.~\eqref{eq:rrbl} holds for $(\g_R,(-)^{[p]_R},R)$ as follows:
\begin{eqnarray*}
R(x^{([p]_R)_R}) &\stackrel{\eqref{eq:rrbl-jac}}{=}&
R\Big(\lambda^{p-1}x^{[p]_R}+\sum_{r=1}^{p-1}\dfrac{1}{r}\lambda^{r-1}\sum_{\alpha\vDash p\atop\ell(\alpha)=r}C_\alpha[x^{\splie_R \alpha}]_R\Big)\\
&=& \lambda^{p-1}R(x)^{[p]}+\sum_{r=1}^{p-1}\dfrac{1}{r}\lambda^{r-1}\sum_{\alpha\vDash p\atop\ell(\alpha)=r}C_\alpha [R(x)^{\rhd_R \alpha}]\\
&\stackrel{\eqref{eq:rrbl-jac}}{=}& R(x)^{[p]_R},
\end{eqnarray*}
where $[x^{\splie_R \alpha}]_R=[x^{\splie_R \alpha_1},[x^{\splie_R \alpha_2},[\cdots[x^{\splie_R \alpha_{r-1}},x^{\splie_R \alpha_r}]_R\cdots]_R]_R]_R$, the second equality uses Eq.~\eqref{eq:rrbl} for $(\g,(-)^{[p]},R)$ and the identity $R(x^{\splie_R k})=R(x)^{\rhd_R k}$ for $k\geq2$.
Hence, $(\g_R,(-)^{[p]_R},R)$ is a restricted Rota-Baxter Lie algebra of weight $\lambda$.
\end{proof}

Also, we want to characterize the restrictability of a Rota-Baxter Lie algebra in characteristic $p$, as a Rota-Baxter enhanced version of Jacobson's classical theorem.

\begin{theorem}\label{thm:restrictability}
	Let $(\g,[-,-],R)$ be a Rota-Baxter Lie algebra of weight $\lambda$ over $\bk$ and $\{e_j\}_{j\in J}$ be a basis of $\lieg$ such that there are $y_j\in \lieg$ with $(\on{ad} e_j)^p=\on{ad} y_j.$
\begin{enumerate}[(1)]
	\item\label{t1} 
There exists exactly one p-mapping $(-)^{[p]}:\lieg\ra \lieg$ such that $e_j^{[p]}=y_j,\ \forall j\in J.$
	\item\label{t2} Furthermore, if $R\left(e_j^{[p]_R}\right)=R(e_j)^{[p]},\forall j\in J,$ where  $(-)^{[p]_R}$ is defined by  Eq. \eqref{eq:rrbl-jac}, then the tuple $(\g,[-,-],(-)^{[p]}, R)$ is a restricted Rota-Baxter Lie algebra of weight $\lambda$.
\end{enumerate}
\end{theorem}
\begin{proof}
	The first statement is due to Jacobson, which can be found in \cite[Theorem 2.3]{SF}. For the second statement, one just need to prove the following claim.
	
	\textbf{Claim:} If  $R\left(e^{[p]_R}\right)=R(e)^{[p]}$, $R\left(f^{[p]_R}\right)=R(f)^{[p]}$ for $e,f\in\g$,  then $R\left((e+f)^{[p]_R}\right)=R(e+f)^{[p]}.$
	
	In fact, first by assumption $R\left((\lambda_je_j)^{[p]_R}\right)=\lambda_j^p R\left(e_j^{[p]_R}\right)=R(\lambda_je_j)^{[p]}$ for every $\lambda_j\in \bk,\ j\in J$.
Then, by the induction on the number of summands, we see that this claim actually guarantees the general case for an arbitrary finite sum, namely
$$R\left(\left(\sum_{j\in J}\lambda_je_j\right)^{[p]_R}\right)
=R\left( \sum_{j\in J}\lambda_je_j\right)^{[p]},$$
which proves the theorem.

Now we focus on the claim. According to Theorem~\ref{thm:rpl-rrb},
for the universal envelope $\U(\g_\lambda)$ of the restricted Lie algebra $\g_\lambda=(\g,[-,-]_\lambda,(-)^{[p]_\lambda})$, we have
\begin{align*}
&y\rhd_R x^{[p]_\lambda} = [\underbrace{x,[x,[\cdots [x}_{p-1},y\rhd_R x]_\lambda\cdots]_\lambda]_\lambda]_\lambda \stackrel{\eqref{eq:der-uea}}{=} y\rhd_R x^p,\\
&x^{*_{\rhd_R} p}-x^{[p]_R} \stackrel{\eqref{eq:p-map-split-uea}}{=}  x^p-x^{[p]_\lambda}
\end{align*}
for any  $x,y\in \g_\lambda$. Define a map $f:\g_R\rightarrow U(\g_R), x\mapsto x^{*_{\rhd_R} p}-x^{[p]_R}. $ Then, given any $j\in J$ and $y\in\g_\lambda$,
\begin{align*}
f(e_j)*_{\rhd_R} y &-y*_{\rhd_R} f(e_j)\\
& \stackrel{\eqref{post-rbb-1}}{=}
(e_j^{*_{\rhd_R} p}-e_j^{[p]_R})y-y(e_j^{*_{\rhd_R} p}-e_j^{[p]_R})
+(e_j^{*_{\rhd_R} p}-e_j^{[p]_R})\rhd_R y - y\rhd_R (e_j^{*_{\rhd_R} p}-e_j^{[p]_R})\\
& \stackrel{\eqref{eq:p-map-split-uea}}{=}
(e_j^p-e_j^{[p]_\lambda})y-y(e_j^p-e_j^{[p]_\lambda})+(e_j^{*_{\rhd_R} p}-e_j^{[p]_R})\rhd_R y - y\rhd_R (e_j^p-e_j^{[p]_\lambda})\\
&=(e_j^{*_{\rhd_R} p}-e_j^{[p]_R})\rhd_R y\ \stackrel{\eqref{Post-2}}{=}\  \alpha^p_{\rhd_R,e_j}(y) - [R(e_j^{[p]_R}),y]\\
&= \alpha^p_{\rhd_R,e_j}(y) - [R(e_j)^{[p]},y]\ \stackrel{\eqref{eq:rl-2}}{=}\ 0,
\end{align*}
where the third equality is due to the fact that $x^p-x^{[p]_\lambda}
\in C_{\U(\g_\lambda)^-}(\g_\lambda)$ for any $x\in\g_\lambda$.
As a result, the descendent Lie algebra $\g_R$ is a Lie subalgebra of the restricted commutator Lie algebra $(U(\g_\lambda)_{\rhd_R}^-,(-)^{*_{\rhd_R} p})$, and
$f$ is a $p$-semilinear map from $\g_R$ to the centralizer $C_{U(\g_\lambda)_{\rhd_R}^-}(\g_R)$, so $(-)^{[p]_R}$ is a $p$-map on $\g_R$ by \cite[Proposition~2.1]{SF}. Namely, $(\g_R,(-)^{[p]_R})$ is a restricted Lie algebra.

Let $\ad_R$ be the adjoint action of $\g_R$ defined by $\ad_Rx(y)\coloneqq [x,y]_R$, and denote $\tilde{S}_i(x,y)\in\g_R$ to express the expansion $\ad_R^{p-1}(tx+y)(x)=\sum_{i=1}^{p-1} i\tilde{S}_i(x,y)t^{i-1}$.  The equality
	$$R\left(\ad_R^{p-1}(tx+y)(x)\right)=\ad^{p-1}\left(tR(x)+R(y)\right)(R(y))$$ implies that
	\[ R\left(\tilde{S}_i(x,y)\right) =S_i\left(R(x),R(y)\right), \quad i=1,\dots,p-1. \]
Using the fact that $(-)^{[p]_R}$ is a $p$-map on $\g_R$, we have
\begin{align*}
R\left( (e+f)^{[p]_R} \right) & \stackrel{\eqref{eq:rl-3}}{=}
R\left(e^{[p]_R}+f^{[p]_R} + \sum_{i=1}^{p-1} \tilde{S}_i(e,f)\right)\\
&=R(e)^{[p]}+R(f)^{[p]}+\sum_{i=1}^{p-1} {S}_i(R(e),R(f))\\
&=\left(R(e)+R(f)\right)^{[p]}=R(e+f)^{[p]},
\end{align*}
which proves the claim.
\end{proof}

\subsection{Two constructions of restricted Rota-Baxter Lie algebras}
Next we provide two natural constructions of restricted Rota-Baxter Lie algebras respectively from Rota-Baxter algebras of arbitrary weight and Rota-Baxter Lie algebras of weight $1$ in prime characteristic.

First we need to give the following combinatorial formula.
\begin{lemma}
Given any two nonnegative integers $j$ and $k$, we have
\begin{equation}\label{eq:binomial}
\sum_{i=0}^{k}(-1)^i{i+j\choose j}{j+k+1\choose i+j+1}=1.
\end{equation}
\end{lemma}
\begin{proof}
The formula can be proven by induction on $k$. It is obvious when $k=0$. If $k>0$, then
\begin{align*}
\sum_{i=0}^{k}(-1)^i{i+j\choose j}{j+k+1\choose i+j+1}&=
\sum_{i=0}^{k}(-1)^i{i+j\choose j}\left({j+k\choose i+j+1}+{j+k\choose i+j}\right)\\
&=\sum_{i=0}^{k-1}(-1)^i{i+j\choose j}{j+k\choose i+j+1}+
\sum_{i=0}^{k}(-1)^i{i+j\choose j}{j+k\choose i+j}\\
&=1+ {j+k\choose j}\sum_{i=0}^{k}(-1)^i{k\choose i}=1,
\end{align*}
where the third equality uses the induction hypothesis and the identity ${i+j\choose j}{j+k\choose i+j}= {k\choose i}{j+k\choose j}$.
\end{proof}

Our first construction says that the commutator Rota-Baxter Lie algebra of a Rota-Baxter algebra of arbitrary weight in prime characteristic is restricted.
\begin{theorem}\label{thm:rba-rrbl}
Let $(A,P)$ be a Rota-Baxter algebra of weight $\lambda$ in characteristic $p$.
Then the triple $(A^-,(-)^p,P)$ is a restricted Rota-Baxter Lie algebra of weight $\lambda$.
\end{theorem}
\begin{proof}
According to Example~\ref{ex:ass-rla}, $(A^-,(-)^p)$ is a restricted Lie algebra. Also, $P$ is clearly a Rota-Baxter operator of weight $\lambda$ on $A^-$. In order to show that $(A^-,(-)^p,P)$ is restricted, we only need to verify that the map $(-)^{[p]_P}$ defined by Eq.~\eqref{eq:rrbl-jac} for $(A^-,(-)^p,P)$ is exactly $(-)^{\star_P p}$, where $\star_P$ is the double product on $A$ defined in Eq.~\eqref{eq:RBA-double}. Indeed,
 if this is true, then Eq.~\eqref{eq:rrbl} clearly holds as follows:
$$P(x^{[p]_P})=P(x^{\star_P p})=P(x)^p,\quad x\in A.$$

When $\lambda=0$, we have $x^{[p]_P}=x^{\rhd_P p}$.
Otherwise, if $\lambda\neq 0$, then $\widetilde P\coloneqq \lambda^{-1}P$ is a Rota-Baxter operator of weight $1$ on $A^-$, so $(A^-,\rhd_{\widetilde P})$ is a post-Lie algebra by Proposition~\ref{prop:split}.
By the universal property of $\U(A^-)$, there exists an algebra map $\pi:\U(A^-)\to A$ such that the following commutative diagram holds:
$$\xymatrix{A^-\ \ar@{^{(}->}[r]^{\!\!\!\!\!\!j}\ar@{=}[rd]_{\rm id} &   \ \U(A^-)
 \ar[d]^{\pi} \\
 & \ A}$$
Hence, we have
\begin{align*}
\nonumber
x^{[p]_P} &\stackrel{\eqref{eq:rrbl-jac}}{=} \lambda^{p-1}x^{p}+\sum_{r=1}^{p-1}\dfrac{1}{r}\lambda^{r-1}\sum_{\alpha\vDash p\atop\ell(\alpha)=r}C_\alpha[x^{\rhd_P \alpha}]
=\lambda^{p-1}\pi\Bigg(x^{p}+\sum_{r=1}^{p-1}\dfrac{1}{r}\sum_{\alpha\vDash p\atop\ell(\alpha)=r}C_\alpha[x^{\rhd_{\widetilde P} \alpha}]\Bigg)\\
&\stackrel{\eqref{eq:glproduct}}{=} \lambda^{p-1}\pi(x^{*_{\rhd_{\widetilde P}}p})\stackrel{\eqref{eq:gl-expand}}{=} \lambda^{p-1}\pi\bigg(\sum_{\alpha\vDash p}C_\alpha x^{\rhd_{\widetilde P} \alpha}\bigg)=\sum_{r=1}^p\lambda^{r-1}\sum_{\alpha\vDash p\atop\ell(\alpha)=r}C_\alpha x^{\rhd_P \alpha},
\end{align*}
which actually also holds even if $\lambda=0$.

Now we check that $x^{[p]_P}=x^{\star_P p}$ for any $x\in A$. First we have
\begin{align*}
x^{\star_P p}&\stackrel{\eqref{eq:RBA-double}}{=}  P(x)x^{\star_P (p-1)}+xP(x^{\star_P (p-1)}) + \lambda xx^{\star_P (p-1)}\\
&= P(x)x^{\star_P (p-1)}+xP(x)^{p-1} + \lambda xx^{\star_P (p-1)}\\
&= \cdots =\sum_{r=1}^{p}\lambda^{r-1}\sum_{j_1+\cdots+ j_{r+1}
=p-r\atop j_1,\dots,j_{r+1}\geq0}P(x)^{j_1}xP(x)^{j_2}\cdots xP(x)^{j_{r+1}}.
\end{align*}
On the other hand, since
\begin{align*}
x^{\rhd_P k}&=[\underbrace{P(x),[P(x),[\cdots [P(x)}_{k-1},x]\cdots ]]]
= \sum_{i=0}^{k-1}{k-1\choose i}(-1)^{k-1-i} P(x)^i x P(x)^{k-1-i},\quad k\geq1,
\end{align*}
we have
\begin{align*}
\sum_{\alpha\vDash p\atop\ell(\alpha)=r}C_\alpha x^{\rhd_P \alpha}
& \stackrel{\eqref{eq:set-part}}{=} \sum_{\alpha_1+\cdots+\alpha_r= p\atop\alpha_1,\dots,\alpha_r\geq1}{\alpha_1-1\choose \alpha_1-1}{\alpha_1+\alpha_2-1\choose \alpha_2-1}\cdots{p-1\choose \alpha_r-1}\sum_{i_1=0}^{\alpha_1-1}\cdots \sum_{i_r=0}^{\alpha_r-1}
{\alpha_1-1\choose i_1}\cdots {\alpha_r-1\choose i_r}\\
 &\quad (-1)^{p-r-i_1-\cdots-i_r}
P(x)^{i_1} x P(x)^{\alpha_1-1-i_1+i_2}x\cdots xP(x)^{\alpha_{r-1}-1-i_{r-1}+i_r} x P(x)^{\alpha_r-1-i_r}\\
&= \sum_{\alpha_1+\cdots+\alpha_r= p\atop\alpha_1,\dots,\alpha_r\geq1}{\alpha_1-1\choose \alpha_1-1}{\alpha_1+\alpha_2-1\choose \alpha_2-1}\cdots{p-1\choose \alpha_r-1}\sum_{j_1+\cdots+j_{r+1}=p-r\atop \alpha_1+\cdots+\alpha_{k-1}-k+1\leq j_1+\cdots+j_k \leq \alpha_1+\cdots+\alpha_k-k,\ k=1,\dots,r}\\
 &\quad
{\alpha_1-1\choose j_1}{\alpha_2-1\choose j_1+j_2-\alpha_1+1}\cdots {\alpha_r-1\choose j_1+\cdots +j_r-\alpha_1-\cdots-\alpha_{r-1}+r-1}\\
 &\quad (-1)^{p+\tfrac{r(r+1)}{2}+\sum\limits_{k=1}^r(r-k+1)j_k+(r-k)\alpha_k}
P(x)^{j_1} x P(x)^{j_2}x\cdots xP(x)^{j_r} x P(x)^{j_{r+1}}\\
&= \sum_{j_1+\cdots+j_{r+1}=p-r\atop j_1,\dots j_{r+1} \geq 0}
\Bigg(\sum_{\alpha_1+\cdots+\alpha_r= p\atop
j_1+\cdots+j_k+k\leq \alpha_1+\dots+\alpha_k\leq j_1+\cdots+j_{k+1}+k,\
k=1,\dots, r}
{\alpha_1-1\choose \alpha_1-1}{\alpha_1+\alpha_2-1\choose \alpha_2-1}\cdots{p-1\choose \alpha_r-1}\\
&\quad
{\alpha_1-1\choose j_1}{\alpha_2-1\choose j_1+j_2-\alpha_1+1}\cdots {\alpha_r-1\choose j_1+\cdots + j_r-\alpha_1-\cdots-\alpha_{r-1}+r-1}\\
 &\quad (-1)^{\sum\limits_{k=1}^r(r-k+1)(\alpha_k-j_k-1)}\Bigg)
P(x)^{j_1} x P(x)^{j_2}x\cdots xP(x)^{j_r} x P(x)^{j_{r+1}}
\end{align*}
for $r=1,\dots,p$, where we denote $j_1=i_1$, $j_k=\alpha_{k-1}-1-i_{k-1}+i_k$ for $k=2,\dots,r$ and $j_{r+1}=\alpha_r-1-i_r$ in the second equality, then conversely $i_k=j_1+\cdots+j_k-\alpha_1-\cdots-\alpha_{k-1}+k-1$ for $k=1,\dots,r$.

Hence, it remains to check that the combinatorial coefficient
\begin{align*}
&\sum_{\alpha_1+\cdots+\alpha_r= p\atop
j_1+\cdots+j_k+k\leq \alpha_1+\dots+\alpha_k\leq j_1+\cdots+j_{k+1}+k,\ k=1,\dots,r}
{\alpha_1-1\choose \alpha_1-1}{\alpha_1+\alpha_2-1\choose \alpha_2-1}\cdots{p-1\choose \alpha_r-1}{\alpha_1-1\choose j_1}{\alpha_2-1\choose j_1+j_2-\alpha_1+1}\\
&\qquad \cdots {\alpha_r-1\choose j_1+\cdots +j_r-\alpha_1-\cdots-\alpha_{r-1}+r-1}(-1)^{\sum_{k=1}^r(r-k+1)(\alpha_k-j_k-1)}
=1
\end{align*}
for any fixed $j_1,\dots, j_{r+1} \geq 0$ such that $j_1+\cdots+j_{r+1}=p-r$, and then
$$x^{[p]_P}=\sum_{r=1}^p\lambda^{r-1}\sum_{\alpha\vDash p\atop\ell(\alpha)=r}C_\alpha x^{\rhd_P \alpha}
=\sum_{r=1}^{p}\lambda^{r-1}\sum_{j_1+\cdots+ j_{r+1}
=p-r\atop j_1,\dots,j_{r+1}\geq0}P(x)^{j_1}xP(x)^{j_2}\cdots xP(x)^{j_{r+1}}=x^{\star_P p}.$$

Let $\beta_k=\sum_{i=1}^k (\alpha_i-j_i-1)$ for $k=1,\dots,r$.
Then, $\beta_k$'s can substitute for $\alpha_k$'s in the above coefficient, with the equivalent constraints that $0\leq \beta_k\leq j_{k+1}$ for $k=1,\dots,r-1$ and $\beta_r=j_{r+1}$. Namely, the coefficient has the form
\begin{align*}
&\sum_{0\leq \beta_k\leq j_{k+1},\ k=1,\dots,r-1\atop
\beta_r=j_{r+1}}
{\beta_2+j_1+j_2+1\choose \beta_2-\beta_1+j_2}\cdots
{\beta_{r-1}+j_1+\cdots+j_{r-1}+r-2\choose \beta_{r-1}-\beta_{r-2}+j_{r-1}}{p-1\choose \beta_r-\beta_{r-1}+j_r}\\
&\qquad {\beta_1+j_1\choose j_1}{\beta_2-\beta_1+j_2\choose j_2-\beta_1}\cdots {\beta_r-\beta_{r-1}+j_r\choose j_r-\beta_{r-1}}(-1)^{\beta_1+\cdots+\beta_r}\\
&=\sum_{0\leq \beta_k\leq j_{k+1},\ k=1,\dots,r-1\atop
\beta_r=j_{r+1}}
{\beta_1+j_1\choose j_1}{\beta_2+j_1+j_2+1\choose j_1+j_2+1}\cdots
{\beta_{r-1}+j_1+\cdots+j_{r-1}+r-2\choose j_1+\cdots+j_{r-1}+r-2}{p-1\choose j_1+\cdots+j_r+r-1}\\
&\qquad {j_1+j_2+1 \choose \beta_1+j_1+1}
{j_1+j_2+j_3+2 \choose \beta_2+j_1+j_2+2}
\cdots {j_1+\cdots+j_r+r-1\choose \beta_{r-1}+j_1+\cdots+j_{r-1}+r-1}(-1)^{\beta_1+\cdots+\beta_r}\\
&=\sum_{\beta_1=0}^{j_2}(-1)^{\beta_1}{\beta_1+j_1\choose j_1}{j_1+j_2+1 \choose \beta_1+j_1+1}
\sum_{\beta_2=0}^{j_3}(-1)^{\beta_2}{\beta_2+j_1+j_2+1\choose j_1+j_2+1}{j_1+j_2+j_3+2 \choose \beta_2+j_1+j_2+2}\\
&\quad \cdots \sum_{\beta_{r-1}=0}^{j_r}(-1)^{\beta_{r-1}}
{\beta_{r-1}+j_1+\cdots+j_{r-1}+r-2\choose j_1+\cdots+j_{r-1}+r-2}{j_1+\cdots+j_r+r-1\choose \beta_{r-1}+j_1+\cdots+j_{r-1}+r-1}\stackrel{\eqref{eq:binomial}}{=}1,
\end{align*}
where the first equality uses the identities
\begin{align*}
{\beta_k+j_1+\cdots+j_k+k-1\choose \beta_k-\beta_{k-1}+j_k}&{\beta_k-\beta_{k-1}+j_k\choose j_k-\beta_{k-1}}\\
&={\beta_k+j_1+\cdots+j_k+k-1\choose j_1+\cdots+j_k+k-1}{j_1+\cdots+j_k+k-1\choose \beta_{k-1}+j_1+\cdots +j_{k-1}+k-1}
\end{align*}
for any $k=2,\dots,r$, and the second equality is due to the fact that ${p-1\choose j_1+\cdots+j_r+r-1}=(-1)^{j_1+\cdots+j_r+r-1}=(-1)^{p-1-j_{r+1}}=(-1)^{\beta_r}$.
So the proof is finished.
\end{proof}

\begin{coro}
Let $(\g,R)$ be a Rota-Baxter Lie algebra of weight $\lambda$ in characteristic $p$. Let $(\U_{\rm RB}(\g),R)$ be the universal enveloping Rota-Baxter associative algebra for $(\g,R)$. Then the triple $(\U_{\rm RB}(\g)^-,(-)^p,R)$ is a restricted Rota-Baxter Lie algebra of weight $\lambda$.
\end{coro}

In \cite[Proposition~4.1]{Do}, Dokas proved that a Rota-Baxter algebra $(A,P)$ of weight 0 in characteristic $p$ induces a restricted pre-Lie algebra structure $(A,\rhd_P,(-)^{\star_p})$ in the sense of Dzhumadil'daev~\cite{Dz}, since $(-)^{\star_p}=(-)^{\rhd_P p}$. Now according to Theorem~\ref{thm:rba-rrbl} and Theorem~\ref{thm:rpl-rrb}, we have the following general result for arbitrary weight.
\begin{coro}\label{coro:rb-rpla}
Let $(A,P)$ be a Rota-Baxter algebra of weight $\lambda$ in characteristic $p$. Then the tuple $(A,\lambda[-,-],\lambda^{p-1}(-)^p,\rhd_P)$ is a restricted post-Lie algebra.
\end{coro}

\begin{remark}
Loday and Ronco defined tridendriform algebras in their study \cite{LR} of polytopes and Koszul duality. Recall that a {\bf tridendriform algebra} $(T,\succ,\prec,\cdot)$ is a $\bk$-vector space endowed with three binary operations $\succ$, $\prec$ and $\cdot$ such that the following relations hold:
\begin{align*}
(x\star y)\succ z &= x\succ(y\succ z),\\
(x\succ y)\prec z   &=x\succ (y\prec z),\\
(x\prec y)\prec z   &=x\prec (y\star z),\\
x\succ(y\cdot z) &= (x\succ y)\cdot z,\\
x\cdot(y\succ z) &= (x\prec y)\cdot z,\\
(x\cdot y)\prec z &= x\cdot (y\prec z),\\
(x\cdot y)\cdot z &= x\cdot (y\cdot z)
\end{align*}
for all elements $x,y,z \in T$, where $x\star y \coloneqq x\succ y+x\prec y+x\cdot y$ is an associative product correspondingly.

According to \cite[Proposition~5.13]{BGN}, for any tridendriform algebra $(T,\succ,\prec,\cdot)$, if one defines the products
\begin{equation}\label{eq:tri-pla}
[a,b]_T\coloneqq a\cdot b-b\cdot a,\quad a\rhd_T b\coloneqq a\succ b-b\prec a,\quad a,b\in T,
\end{equation}
then $(T,[-,-]_T,\rhd_T)$ is a post-Lie algebra.

On the other hand, a Rota-Baxter algebra $(A,P)$ of weight $\lambda$ provides a tridendriform algebra $(A,\succ_P,\prec_P,\cdot_\lambda)$~(see \cite{Eb}), where
$$x\succ_P y\coloneqq P(x)y,\quad x\prec_P y\coloneqq xP(y),\quad x\cdot_\lambda y\coloneqq \lambda xy,\quad x,y\in A.$$
The corresponding post-Lie algebra structure defined by Eq.~\eqref{eq:tri-pla} is exactly the splitting post-Lie algebra $(A,\lambda[-,-],\rhd_P)$ of the commutator Rota-Baxter Lie algebra $(A^-,P)$. By Corollary~\ref{coro:rb-rpla}, we further know that $(A,\lambda[-,-],\lambda^{p-1}(-)^p,\rhd_P)$ is a restricted post-Lie algebra.

Consequently, we expect that a tridendriform algebra $(T,\succ,\prec,\cdot)$ in characteristic $p$ induces a restricted post-Lie algebra $(T,[-,-]_T,(-)^p,\rhd_T)$
in general.
\end{remark}

In order to give the second construction of restricted Rota-Baxter Lie algebras, we recall the following prominent result.
According to Goncharov's work in \cite{Go}, a Rota-Baxter operator $R$ of weight $1$ on $\g$ can be extended onto the whole universal envelope $\U(\g)$ properly by letting
\begin{equation}\label{eq:rb-uea}
R(xy)=R(x)R(y)-R([R(x),y]),\quad x\in \g,\,y\in \U(\g).
\end{equation}
It leads to the notion of a {\it cocommutative} {\bf Rota-Baxter Hopf algebra} consisting of a cocommutative Hopf algebra $H=(H,\cdot,\Delta,\vep,S)$ and a coalgebra endomorphism $R$ of $H$ such that
\begin{equation}\label{eq:rb-hopf}
R(x)R(y)=R(x_1R(x_2)yS(R(x_3))),\quad x,y\in H.
\end{equation}
Moreover, there is a {\bf descendent Hopf algebra} $H_R=(H,*_R,\Delta,\vep,S_R)$ defined by
\begin{equation*}
x*_R y=x_1R(x_2)yS(R(x_3)),\quad S_R(x)=S(x_1R(x_2))R(x_3),\quad x,y\in H.
\end{equation*}
Therefore, $ R(x*_R y)=R(x)R(y)$ for any $x,y\in H$.

It is pointed out in \cite[Theorem~3.6]{LST} that the Rota-Baxter Hopf algebra $(\U(\g),R)$ induces a post-Hopf algebra $(\U(\g),\rhd_R)$, and the descendent Hopf algebra $\U(\g)_R$ is the corresponding subjacent Hopf algebra $\U(\g)_{\rhd_R}$, namely $*_R=*_{\rhd_R}$ and $S_R=S_{\rhd_R}$.

\begin{theorem}\label{thm:rrbl}
Given a Rota-Baxter Lie algebra $(\g,[-,-],R)$ of weight $1$ over a field $\bk$ of characteristic $p$, let $\g^{(p)}$ be the primitive Lie subalgebra of $\U(\g)^-$ as in Example~\ref{ex:uea-rla}. Then $(\g^{(p)},[-,-],(-)^p,R)$ is a restricted Rota-Baxter Lie algebra of weight $1$, where we denote
the restriction $R|_{\g^{(p)}}$ of the Rota-Baxter operator $R$ of $\U(\g)$ still as $R$ without confusion.
\end{theorem}

\begin{proof}
First according to \cite[Theorem~3.2]{LST},
the restriction $R|_{\g^{(p)}}$ is a Rota-Baxter operator of weight $1$ on the primitive Lie subalgebra $\g^{(p)}$ of the Rota-Baxter Hopf algebra $(\U(\g),R)$.

On the other hand, $(\g^{(p)},[-,-],(-)^p)$ is a restricted Lie algebra in Example~\ref{ex:uea-rla}, namely the $p$-map $(-)^{[p]}=(-)^p$. So we just have $(-)^{[p]_R}=(-)^{*_{\rhd_R} p}=(-)^{*_R p}$ by Eqs.~\eqref{eq:glproduct} and \eqref{eq:rrbl-jac}, then
$$R(x^{[p]_R})\ =\ R(x^{*_R p}) \ \stackrel{\eqref{eq:rb-hopf}}{=}\ R(x)^p,\quad x\in\g^{(p)}.$$
Hence, $(\g^{(p)},[-,-],(-)^p,R)$ is a restricted Rota-Baxter Lie algebra of weight $1$.
\end{proof}

\begin{theorem}\label{thm:rrbl-uea}
Given a restricted Rota-Baxter Lie algebra $(\g,[-,-],(-)^{[p]},R)$ of weight $1$, the Rota-Baxter operator on $\U(\g)$ can descend to a Rota-Baxter operator on the restricted enveloping algebra $\uu(\g)$.
Furthermore, if $R$ is idempotent and $\bk$ is not of characteristic $2$, then such a Rota-Baxter operator on $\uu(\g)$ is a Hopf algebra endomorphism of $\uu(\g)$.
\end{theorem}
\begin{proof}
Since the descendent product $*_R$ of $\U(\g)_R$ coincides with the generalized Grossman-Larson product $*_{\rhd_R}$ of $\U(\g)_{\rhd_R}$, we have
$$R(x^p-x^{[p]})\ \stackrel{\eqref{eq:p-map-uea}}{=}\ R(x^{*_R p}-x^{[p]_R})
\ \stackrel{\eqref{eq:rrbl},\,\eqref{eq:rb-hopf}}{=}\ R(x)^p-R(x)^{[p]},\quad x\in\g.$$
Also, $x^p-x^{[p]},\,x\in \g$, is a central primitive element in $\U(\g)$, thus the ideal $I=(x^p-x^{[p]}\,|\,x\in \g)$ is stable under $R$ by Eq~\eqref{eq:rb-uea}. So $R$ can be induced as a Rota-Baxter operator on the restricted enveloping algebra $\uu(\g)$.

If $R$ is an idempotent Rota-Baxter operator of weight $1$ on $\g$ and $\bk$ is not of characteristic $2$, then $R$ is a Lie algebra endomorphism of $\g$ such that $R([R(x),y])=0,\ x,y\in \lieg$ by Proposition~\ref{prop:idempotent-rb}. Now by induction on the length of words in $\U(\g)$, Eq.~\eqref{eq:rb-uea} implies that the extended Rota-Baxter operator $R$ on $\U(\g)$ is an algebra endomorphism. Since $R$ is also a coalgebra endomorphism of $\U(\g)$, it is a Hopf algebra endomorphism, so is the induced Rota-Baxter operator on $\uu(\g)$.
\end{proof}

At the end of this subsection, we provide an expansion formula of the restricted Rota-Baxter operator $R$ on $\g^{(p)}$ given in Theorem~\ref{thm:rrbl}.
In \cite[Theorem~3.5]{Li}, the first author obtained the closed inverse formula for the Guin-Oudom isomorphism $\phi:\U(\g_\rhd)\to \U(\g)_\rhd$ of a post Lie algebra $(\g,[-,-],\rhd)$. It actually implies the following expression of powers in $\U(\g)$ in terms of the generalized Grossman-Larson product $*_\rhd$:
\begin{align}
\label{eq:OG-inverse}
x^n=\sum_{\alpha\vDash n}(-1)^{n-r}C_\alpha x^{\btr \alpha_1}*_\rhd\cdots*_\rhd x^{\btr \alpha_r},\quad x\in\g^{(p)},
\end{align}
where the sum ranges over all compositions $\alpha=(\alpha_1,\dots,\alpha_r)$ of $n$, the coefficient $C_\alpha$ has been defined in Eq.~\eqref{eq:set-part} and
$x^{\btr k}$ is the sum of all $(k-1)!$ iterated post-Lie products of $k$ $x$'s (counting multiplicities). For example,
$$x^{\btr 4}=x\rhd(x\rhd(x\rhd x)) + x\rhd((x\rhd x)\rhd x) + 2(x\rhd x)\rhd(x\rhd x) + (x\rhd(x\rhd x))\rhd x + ((x\rhd x)\rhd x)\rhd x.$$

\begin{prop}
For the restricted Rota-Baxter Lie algebra $(\g^{(p)},[-,-],(-)^p,R)$ of weight $1$ in Theorem~\ref{thm:rrbl}, the following expansion formula of $R$ on $\g^{(p)}$ holds:
\begin{equation*}
R(x^p)=R(x)^p+\sum_{r=1}^{p-1}(-1)^{p-r}\dfrac{1}{r}\sum_{\alpha\vDash p\atop\ell(\alpha)=r} C_\alpha [R(x^{\btr_R \alpha_1}),[\cdots[R(x^{\btr_R \alpha_{r-1}}),R(x^{\btr_R \alpha_r})]\cdots]],\quad x\in \g^{(p)},
\end{equation*}
where $x^{\btr_R k}$ is the sum of all $(k-1)!$ iterated post-Lie products of $k$ $x$'s (counting multiplicities) with respect to $\rhd_R$.
\end{prop}
\begin{proof}
For any $x\in \g^{(p)}$,  we clearly have $R(x^p)\in \g^{(p)}$. Since $*_{\rhd_R}=*_R$,
\begin{eqnarray*}
R(x^p) &\stackrel{\eqref{eq:OG-inverse}}{=}&
R\left(x^{*_R p}\right) +\sum_{r=1}^{p-1}(-1)^{p-r}\sum_{\alpha\vDash p\atop\ell(\alpha)=r}C_\alpha R\left(x^{\btr_R \alpha_1}*_R\cdots*_R x^{\btr_R \alpha_r}\right)\\
&\stackrel{\eqref{eq:rb-hopf}}{=}& R(x)^p +\sum_{r=1}^{p-1}(-1)^{p-r}\sum_{\alpha\vDash p\atop\ell(\alpha)=r}C_\alpha R(x^{\btr_R \alpha_1})\cdots R(x^{\btr_R \alpha_r}).
\end{eqnarray*}
So  we get the desired formula, by applying the Dynkin-Specht-Wever Theorem in characteristic $p$ to the last equality.
\end{proof}

\subsection{Rota-Baxter $p$-envelopes of a Rota-Baxter Lie algebra}
Referring to the construction of $p$-envelopes of a post-Lie algebra in Section~\ref{sec:rpla}, we also study $p$-envelopes of a Rota-Baxter Lie algebra.
\begin{defn}
	Let $(\lieg,(-)^{[p]},R)$ be a restricted Rota-Baxter Lie algebra over $\bk$.
	A Rota-Baxter subalgebra $\lieh\subseteq \g$
is called a {\bf Rota-Baxter $p$-subalgebra} if $x^{[p]}\in \lieh, \forall x\in \lieh$. 
	
	For a subset $E\subseteq \g$, the intersection of all Rota-Baxter $p$-subalgebras containing $E$ will be denoted by $E_p$, which is a Rota-Baxter $p$-subalgebra of $\g$ and is referred to as the Rota-Baxter $p$-subalgebra generated by $E$ in $\g.$
\end{defn}	

\begin{defn}
Let $(\g, R)$ be a Rota-Baxter Lie algebra over $\bk$.

A tuple $(\liel, (-)^{[p]'}, R', i)$ consisting of a restricted Rota-Baxter Lie algebra $(\liel, (-)^{[p]'}, R')$ and a Rota-Baxter Lie algebra homomorphism $i:\g\to \liel$ is called a {\bf Rota-Baxter $p$-envelope} of $\g$ if $i$ is injective and $i(\g)_p=\liel.$
	
A Rota-Baxter $p$-envelope $(\liel, (-)^{[p]'}, R', i)$ of $\g$ is called {\bf universal} if it satisfies the following universal property: for any restricted Rota-Baxter Lie algebra $\liec$ and Rota-Baxter Lie algebra homomorphism $f:\g\to \liec$, there exists a unique restricted Rota-Baxter Lie algebra homomorphism $g: \liel\to \liec$ such that $g\circ i=f.$
\end{defn}	

\begin{theorem}\label{thm:urbe}
Every Rota-Baxter Lie algebra $(\g,R)$ of weight $1$ over $\bk$ has a universal Rota-Baxter $p$-envelope $(\g^{(p)},(-)^p,R,i)$, where $i:\g\to \g^{(p)}$ is the natural embedding.
\end{theorem}
\begin{proof}
By Theorem \ref{thm:rrbl}, $\g^{(p)}\subset \U(\g)$ is a restricted Rota-Baxter Lie algebra of weight $1$ containing $\g$. Since $\g^{(p)}$ is the $p$-subalgebra of $\U(\g)^-$ generated by $\g$, it is clear that $(\g^{(p)},(-)^p,R,i)$ is a Rota-Baxter $p$-envelope of $\g$.
	
Now, for each restricted Rota-Baxter Lie algebra $(\liel,(-)^{[p]'}, R')$ and a Rota-Baxter Lie algebra homomorphism $f:\g \to \liel$, we can apply Theorem~\ref{thm:rrbl-uea} to $(\liel,(-)^{[p]'}, R')$ and use the universal property of $\U(\g)$ to obtain a Rota-Baxter Hopf algebra homomorphism $\bar{f}: \U(\g)\to \uu(\liel)$ such that $\bar f\circ i=f$.
Then, similar to the end of the proof in Theorem~\ref{thm:pla-rpla}, we see that $g\coloneqq \bar{f}|_{\g^{(p)}}: \g^{(p)} \to \liel$ is a restricted Rota-Baxter Lie algebra homomorphism such that $g\circ i=f$, and $(\g^{(p)},(-)^p,R,i)$ is a universal Rota-Baxter $p$-envelope of $(\g,R)$.
\end{proof}

\vspace{0.1cm}
 \noindent
{\bf Acknowledgements.} We would like to thank Honglei Lang for inspiring us to give the graph subalgebra characterization of restricted Rota-Baxter operators. This work is supported by National Natural Science Foundation of China (12071094, 12461005), Basic and Applied Basic Research Foundation of Guangdong Province (2026A1515012750), and Basic Research Program of Yunnan Province (202601AT070218).

\bibliographystyle{amsplain}

\end{document}